\documentclass[bibother]{asl}

\usepackage{verbatim}
\usepackage{tikz}
\usetikzlibrary{trees}
\usepackage{cite}

\theoremstyle{plain}
\newtheorem{theorem}{Theorem}[section]
\newtheorem{lemma}[theorem]{Lemma}

\newtheorem{cor}[theorem]{Corollary}

\theoremstyle{plain}
\newtheorem*{Thm:FMforQO}{Theorem \ref{thm:FMforQO}}
\newtheorem*{Thm:qltomegaUniversal}{Theorem \ref{thm:qltomegaUniversal}}
\newtheorem*{Thm:qlttreeUniversal}{Theorem \ref{thm:TreesUniversal}}
\newtheorem*{Thm:EmGroupsUniversal}{Theorem \ref{thm:EmGroupsUniversal}}
\newtheorem*{Thm:FGGroupEmbedUniversal}{Theorem \ref{thm:FGGroupEmbedUniversal}}

\theoremstyle{definition}
\newtheorem{definition}[theorem]{Definition}
\newtheorem*{problem}{Open Problem}

\theoremstyle{remark}
\newtheorem{remark}[theorem]{Remark}

\numberwithin{equation}{section}

\newcommand{\F}{\mathbb{F}}
\newcommand{\N}{\mathbb{N}}

\newcommand{\Z}{\mathbb{Z}}

\newcommand{\mc}{\mathcal}
\newcommand{\from}{\colon}  
\newcommand{\into}{\hookrightarrow}
\newcommand{\Iff}{\Leftrightarrow}
\newcommand{\<}{\langle}
\renewcommand{\>}{\rangle}
\newcommand{\qlt}{\preccurlyeq}

\newcommand{\Sg}{\operatorname{Sg}}
\newcommand{\LongIff}{\;\;\Longleftrightarrow \;\;}

\title{Universal Countable Borel Quasi-Orders}
\author{Jay Williams}
\revauthor{Williams, Jay}
\address{Department of Mathematics\\ California Institute of Technology\\ Pasadena, CA 91125}
\email{jaywill@caltech.edu}
\date{}

\begin{document}

\begin{abstract}

In recent years, much work in descriptive set theory has been focused on the Borel complexity of naturally occurring classification problems, in particular, the study of countable Borel equivalence relations and their structure under the quasi-order of Borel reducibility.  Following the approach of Louveau and Rosendal for the study of analytic equivalence relations, we study countable Borel quasi-orders.

In this paper we are concerned with universal countable Borel quasi-orders, i.e.\ countable Borel quasi-orders above all other countable Borel quasi-orders with regard to Borel reducibility.  We first establish that there is a universal countable Borel quasi-order, and then establish that several countable Borel quasi-orders are universal.  An important example is an embeddability relation on descriptive set theoretic trees.

Our main result states that embeddability of finitely generated groups is a universal countable Borel quasi-order, answering a question of Louveau and Rosendal.  This immediately implies that biembeddability of finitely generated groups is a universal countable Borel equivalence relation.  The same techniques are also used to show that embeddability of countable groups is a universal analytic quasi-order.

Finally, we show that, up to Borel bireducibility, there are $2^{\aleph_0}$ distinct countable Borel quasi-orders which symmetrize to a universal countable Borel equivalence relation.

\end{abstract}

\maketitle

\section{Introduction}

A countable Borel quasi-order $Q$ is a quasi-order defined on a Polish space $X$ (or more generally, a standard Borel space) such that $Q$ is Borel when viewed as a subset of $X^2$ and for every $x\in X$, the set of predecessors of $x$, $\{y \mid y \mathbin{Q} x\}$, is countable.  There are several natural examples, such as $\leq_T$ and $\leq_1$ on $2^\N$, as well as the embeddability relation on the space $\mc G$ of finitely generated groups.

As in the case of Borel equivalence relations, we are largely interested in how countable Borel quasi-orders are related to each other under Borel reducibility.

\begin{definition}\hfill

a)  Suppose that $Q$ is a Borel quasi-order on a Polish space $X$ and $Q'$ is a Borel quasi-order on a Polish space $Y$.  We say that $Q$ is Borel reducible to $Q'$, written $Q \leq_B Q'$, if there is a Borel function $f\from X\to Y$ such that
$$x \mathbin{Q} y \Iff f(x) \mathbin{Q'} f(y)$$

b)  A countable Borel quasi-order $Q'$ is universal if for every countable Borel quasi-order $Q$, $Q  \leq_B Q'$.
\end{definition}

Borel reducibility is intended to capture the notion of the relative complexity of quasi-orders, so that if $Q\leq_B Q'$ then we consider $Q'$ to be more complicated than $Q$.  A universal countable Borel quasi-order can then be thought of as a countable Borel quasi-order which is as complicated as possible.  It was shown by Dougherty, Jackson, and Kechris in \cite{DJK94} that there is a universal countable Borel equivalence relation, and in \cite{LR05} Louveau and Rosendal showed there is a universal analytic quasi-order.  In section \ref{sec:UniversalExist}, using an analogue of the Feldman-Moore Theorem, we prove the following result.

\begin{Thm:qltomegaUniversal}
There is a universal countable Borel quasi-order.
\end{Thm:qltomegaUniversal}

Given a quasi-order $Q$ on $X$, the corresponding equivalence relation $E_Q$ on $X$ is defined by
$$ x E_Q y \Iff x Q y \, \wedge \, y Q x $$
For example, Turing equivalence $\equiv_T$ is $E_{\leq_T}$.  It is easily checked that if $Q$ is a universal countable Borel quasi-order, then $E_Q$ is a universal countable Borel equivalence relation, and this provides another source of examples for such equivalence relations.

Much of this paper is dedicated to proving that various countable Borel quasi-orders are universal.  One universal countable Borel quasi-order in particular is used in several of these proofs.  Recall that given a discrete space $X$, a subset $T\subseteq X^{<\N}$ is said to be a tree on $X$ if it is closed under initial segments.

\begin{definition}\label{def:treeqo}
Given a discrete space $X$, the quasi-order $\qlt_X^{tree}$ on the space of trees on $X$ is defined by
$$ T\qlt_X^{tree} T' \Iff (\exists u\in X^{<\N})\hspace{4pt} T=T'_u $$
where $T'_u=\{v\in X^{<\N} \mid u^\frown v\in T'\}$.  (Here $u^\frown v$ indicates the string $u$ followed by $v$.)
\end{definition}

\begin{Thm:qlttreeUniversal}
$\qlt_2^{tree}$ is a universal countable Borel quasi-order.
\end{Thm:qlttreeUniversal}

This quasi-order is combinatorially simple and thus easy to work with.  This makes it useful for establishing other quasi-orders are universal.

Our ultimate goal is to show the biembeddability relation for finitely generated groups is a universal countable Borel equivalence relation.  Although it may be possible to prove this only using results on countable Borel equivalence relations, the use of quasi-orders seems to be the most direct route to this result.  We first use small cancellation methods from group theory to establish the following theorem.

\begin{Thm:EmGroupsUniversal}
Embeddability of countable groups is a universal analytic quasi-order.
\end{Thm:EmGroupsUniversal}

The ideas developed in this proof are then used to reduce $\qlt_2^{tree}$ to the embeddability relation for finitely generated groups.  Thus we have the following theorem.

\begin{Thm:FGGroupEmbedUniversal}
Embeddability of finitely-generated groups is a universal countable Borel quasi-order.
\end{Thm:FGGroupEmbedUniversal}

Thus the embeddability structure of finitely-generated groups is a complicated as possible.  We find as a corollary that the bi-embeddability relation for finitely-generated groups is a universal countable Borel equivalence relation.  This answers a question of Louveau and Rosendal in \cite{LR05}.  It was previously shown in \cite{TV99} that the isomorphism relation for finitely-generated groups is a universal countable Borel equivalence relation, and so this result can be seen as saying the two relations have precisely the same complexity.

The rest of the paper is organized as follows.  In section \ref{sec:UniversalExist}, we establish that there is  a universal countable Borel quasi-order by proving a Feldman-Moore-type result relating countable Borel quasi-orders to Borel actions of countable monoids.  In section \ref{sec:TreesQO}, we show that $\qlt_2^{tree}$ is universal.  In section \ref{sec:GroupUniversal}, we show that various group-theoretic countable Borel quasi-orders, which arise as simple generalizations of well-known universal countable Borel equivalence relations, are universal.  In sections \ref{sec:CtblGroupEmbed} and \ref{sec:FGGroupEmbed}, we prove the theorems on the embeddability relations for countable groups and finitely generated groups.  Finally, in section \ref{sec:Final} we show that, up to Borel bireducibility, there are $2^{\aleph_0}$ countable Borel quasi-orders which give rise to a universal countable Borel equivalence relation.

\section{A universal countable Borel quasi-order}\label{sec:UniversalExist}

We start by proving an analogue of the Feldman-Moore Theorem \cite{FM77} for countable Borel quasi-orders.  Although to the author's knowledge, this result is not in the literature, it is a straightforward application of the well-known Lusin-Novikov theorem (see Theorem 18.10 in Kechris \cite{K95}), and should perhaps be considered as folklore.

\begin{theorem}\label{thm:FMforQO}
If $\qlt$ is a countable Borel quasi-order on the Polish space $X$, there is a monoid $M$ which acts on $X$ in a Borel way such that
$$x\qlt y \LongIff (\exists m\in M) \hspace{4pt} x = m\cdot y. $$
\end{theorem}
\begin{proof}
First, note that by definition for all $y\in X$, $\qlt_y = \{ x \mid x\qlt y\}$ is countable, which implies the set $\qlt\subseteq X\times X$ has countable sections with respect to its second coordinate.  By the Lusin-Novikov theorem, $\succcurlyeq\,= \cup_n f_n$, where each $f_n\from E_n\to X$ is a Borel function, with $E_n\subseteq X$ Borel.

We can extend these to functions defined on all of $X$ by letting $f_n(y)=y$ for $y\in\nolinebreak X\setminus\nolinebreak E_n$.  These functions are still Borel, and their union is still equal to $\succcurlyeq$ by reflexivity.  We may also add the identity function to our collection without changing the union, again by reflexivity.  With all this in place, the $f_n$ generate a monoid $M$ under composition, and $M$ acts on $X$ by $m\cdot x=m(x)$.  If $x\qlt y$ then there exists $m\in M$ such that $x=m\cdot y=m(y)$, and the transitivity of $\qlt$ ensures that for all $m\in M$ and $x\in X$, $m\cdot x\qlt x$.
\end{proof}

We wish to use this result to show that there is a universal countable Borel quasi-order.  Our approach closely follows the proof of Dougherty, Jackson, and Kechris in \cite{DJK94} that there is a universal countable Borel equivalence relation.

\begin{definition}\label{def:qlt_M^X}
For every standard Borel space $X$ and countable monoid $M$, the corresponding canonical Borel action of $M$ on $X^M$ is defined by $(m\cdot f)(s)=f(sm)$ for $m,s\in M$ and $f\in X^M$.  We denote the corresponding quasi-order by $\qlt_M^X$, i.e.\ for $f,g\in X^M$,
$$f\qlt_M^X g \LongIff (\exists m\in M)\hspace{4pt} f=m\cdot g.$$
\end{definition}

To see that this is an action, let $m,n\in M$ and $f\in X^M$.  Then
\begin{align*}
(m\cdot(n\cdot f))(s) &= (n\cdot f)(sm) \\
 &= f(smn) \\
 &= (mn\cdot f)(s)
\end{align*}
as desired.

Note that in the above definition, if $M$ is a group, then $\qlt_M^X$ is in fact an equivalence relation.  In particular, $\qlt_{\F_2}^2$ is the universal countable Borel equivalence relation $E_\infty$.

\begin{definition}[The quasi-order $\qlt_\omega$]\label{def:qlt_omega}
Let $M_\omega$ be the free monoid on countably many generators.  Then define $\qlt_\omega$ to be $\qlt_{M_\omega}^{2^\N}$.
\end{definition}

\begin{theorem}\label{thm:qltomegaUniversal}
$\qlt_\omega$ is a universal countable Borel quasi-order.
\end{theorem}
\begin{proof}
Let $\qlt$ be a countable Borel quasi-order on a Polish space $X$.  By Theorem \ref{thm:FMforQO}, there is a countable monoid $M$ such that $\qlt$ is the quasi-order induced by a Borel action of $M$ on $X$.  Let $f\from M_\omega\to M$ be a surjective homomorphism.  Then we can define an action of $M_\omega$ on $X$ by
$$m\cdot x = f(m)\cdot x. $$
This action is Borel and also induces $\qlt$, and so without loss of generality we may assume that $M=M_\omega$.

Let $\{U_i\}_{i\in\N}$ be a sequence of Borel sets in $X$ which separates points.  Then we define $\phi\from X\to\nolinebreak (2^\N)^{M_\omega}$ by $x\mapsto \phi_x$, with
$$\phi_x(s)(i)=1 \LongIff s\cdot x\in U_i. $$
This map is Borel, and since the $U_i$ separate points, we see it is injective.  Furthermore, if $t\in M_\omega$, then $t\cdot \phi_x = \phi_{t\cdot x}$.  To see this, let $s\in M_\omega$, and $i\in \N$.  Then
\begin{align*}
\phi_{t\cdot x}(s)(i)=1 &\LongIff s\cdot t\cdot x\in U_i \\
 &\LongIff \phi_x(st)(i)=1 \\
 &\LongIff t\cdot \phi_x(s)(i)=1
\end{align*}

Now suppose that $x\qlt y$.  Then there exists $m\in M_\omega$ such that $x = m\cdot y$.  It follows that $\phi_x=\phi_{m\cdot y}=m\cdot \phi_y$, and so $\phi_x\qlt_\omega \phi_y$.  The same reasoning works in reverse, and hence $\phi_x\qlt_\omega \phi_y$ implies that $x\qlt y$.  Thus $\phi$ is a Borel reduction.
\end{proof}

Thus there exists a universal countable Borel quasi-order.  Next we wish to find universal countable Borel quasi-orders which are easier to work with.  We proceed by a series of easily proven lemmas which are the analogues of propositions 1.4-1.8 in Dougherty, Jackson, and Kechris \cite{DJK94}.  The proofs of most of them are virtually the same, and so we omit them here.

\begin{lemma}\label{prop:HomImage}
If $M,N$ are monoids and $M$ is a homomorphic image of $N$, then $\qlt_M^X\leq_B\,\qlt_N^X$.
\end{lemma}

\begin{lemma}\label{prop:2to3}
For any countable monoid $M$, $\qlt_M^{2^{\Z-\{0\}}}\leq_B\, \qlt_{M\times\Z}^3$.
\end{lemma}

\begin{lemma}\label{prop:3to2}
For any countable monoid $M$, $\qlt_M^3\leq_B\qlt_{M\times\Z_2}^2$.
\end{lemma}

\begin{lemma}\label{omegato2}
Let $M_2$ denote the free monoid on 2 generators.  Then $$\qlt_{M_\omega}^2\, \leq_B \, \qlt_{M_2}^2.$$
\end{lemma}
\begin{proof}
We start by embedding $M_\omega$ into $M_2$ in order to view it as a submonoid of $M_2$.  Let $M_\omega=\< x_1,x_2,\ldots \>$ and $M_2=\<a,b\>$.  We define our embedding by $e\mapsto e$ and $x_n\mapsto ab^n$ for all $n\in \N^+$.

Next we note that if $h\in M_2$, then we can canonically write $h$ as a product $h=h'g$, with $g\in M_\omega$ and $h'\in M_2\setminus M_\omega$, possibly with $g=e$ or $h'=e$, by finding the longest word in $M_\omega$ at the end of $h$.  Define $L\from M_2\to\N$ by
$$L(h)=\text{ the length of }h'$$
where $h=h'g$ is the canonical form of $h$.  This function has the desirable property that multiplying an element $h\in M_2$ on the right by an element $g\in M_\omega$ does not change the given length, i.e.\ $L(h)=L(hg)$.

Define $f\from 2^{M_\omega}\to 2^{M_2}$ by $p\mapsto p^*$ where
$$p^*(h)=\begin{cases}
p(h) & \text{if $L(h)=0$} \\
1 & \text{if $L(h)=1$} \\
0 & \text{if $L(h)>1$}
\end{cases}.$$
Suppose that $p\qlt_{M_\omega}^2\, q$.  Then $\exists g\in M_\omega$ such that $p=g\cdot q$.  So if $h\in M_\omega$,
\begin{align*}
(g\cdot q^*)(h) &= q^*(hg) \\
 &= q(hg) \\
 &=(g\cdot q)(h) \\
 &=p(h) \\
 &=p^*(h)
\end{align*}
If $h\in M_2\setminus M_\omega$, then since $L(h)=L(hg)$, we find
\begin{align*}
(g\cdot q^*)(h) &=q^*(hg) \\
 &=p^*(h)
\end{align*}
So $p^*\qlt_{M_2}^X\, q^*$.

Now suppose that $p^*\qlt_{M_2}^2\, q^*$.  Then there exists $g\in M_2$ such that $p^*=g\cdot q^*$.  Clearly if $g\in M_\omega$, then $p=g\cdot q$.  If instead $g\in M_2\setminus M_\omega$, then $L(g)\geq 1$, and we should have that $p^*(b)=(g\cdot q^*)(b)=q^*(bg)$.  But $p^*(b)=1$, while $L(bg)>1$, and so $q^*(bg)=0$.  Thus this case cannot happen, and hence $f$ is a Borel reduction.
\end{proof}

\begin{theorem}\label{thm:qltM2Universal}
$\qlt_\omega \, \leq_B \, \qlt_{M_2}^2$.  It follows that $\qlt_{M_2}^2$ is universal.
\end{theorem}
\begin{proof}
Using the preceding lemmas, we find that
\begin{align*}
\qlt_{M_\omega}^{2^\N} &\leq_B\; \qlt_{M_\omega}^{2^{\Z-\{0\}}} \\
 &\leq_B\; \qlt_{M_\omega\times\Z}^3 \text{ by Prop. \ref{prop:2to3}}\\
 &\leq_B\; \qlt_{M_\omega\times\Z\times\Z_2}^2 \text{ by Prop. \ref{prop:3to2}}\\
 &\leq_B\; \qlt_{M_\omega}^2 \text{ by Prop. \ref{prop:HomImage}}\\
 &\leq_B\; \qlt_{M_2}^2 \text{ by Prop. \ref{omegato2}}  
\end{align*}
\end{proof}

The quasi-order $\qlt_{M_2}^2$ is easier to work with than $\qlt_\omega$, as both the monoid and the space being acted on are simpler.  Using $\qlt_{M_2}^2$, we will find another universal countable Borel quasi-order, this one of a more combinatorial nature.

\section{A quasi-order on trees}\label{sec:TreesQO}

In this section, we will reduce $\qlt_{M_2}^2$ to $\qlt_2^{tree}$, the quasi-order on descriptive-set-theoretic trees from definition \ref{def:treeqo}.  This has the advantage of moving us away from working with monoids and towards more classical areas of mathematics.  We must first make a few intermediate reductions.

\begin{definition}
The quasi-order $\qlt_2^s$ (the $s$ is for ``suffix") on $\mc P(M_2)$ is defined by
$$A\qlt_2^s B \LongIff (\exists m\in M_2)\hspace{4pt} Am=B^m$$ where
$$B^m=B\cap M_2m. $$
\end{definition}

\begin{remark}
Note that if we made a similar definition for a group $M$, then we would always have that $B^m=B$, since in this case $Mm=M$.  So this definition is only interesting when dealing with a monoid.
\end{remark}

If we identify $\mc P(M_2)$ with $2^{M_2}$, then this quasi-order is the same as $\qlt_{M_2}^2$.  Writing it in this way brings out the fact that knowing a set $A\in\mc P(M_2)$ and that \mbox{$A\qlt_{M_2}^2 B$} only gives partial information about $B$.  This differs from $E_\infty$, the analogous equivalence relation, since knowing $A\in\mc P(\F_2)$ and that $A\mathbin{E_\infty} B$ gives information about all of $B$.

Next, we modify this quasi-order slightly, in order to make it somewhat easier to work with.

\begin{definition}
The quasi-order $\qlt_2^p$ (the $p$ is for ``prefix") on $\mc P(M_2)$ is defined by
$$A\qlt_2^p B \LongIff (\exists m\in M_2)\hspace{4pt} mA=B_m$$ where
$$B_m=B\cap mM_2. $$
\end{definition}

As before, this definition is only interesting when working with a monoid.

\begin{theorem}
$\qlt_2^s\,\sim_B\,\qlt_2^p$
\end{theorem}
\begin{proof}
Every nontrivial element $w\in M_2$ may be written as $w=a^{n_0}b^{m_0}\ldots a^{n_k}b^{m_k}$, where $n_i,m_j\in\N$, and only $n_0$ or $m_k$ may be 0.  Define $\bar w=b^{m_k}a^{n_k}\ldots b^{m_0}a^{n_0}$, and $\bar e=e$.  Then the bijection $f\from M_2\to M_2$ defined by $f(w)=\bar w$ induces a Borel bijection $f^*\from\mc P(M_2)\nolinebreak\to\nolinebreak\mc P(M_2)$ such that if $Am=B^m$, then $\bar m f^*(A)=f^*(B)_{\bar m}$.  Similarly, if $wf^*(A)=f^*(B)_w$, then $A\bar w=B^{\bar w}$.  Thus $f^*$ is a Borel reduction from $\qlt_2^s$ to $\qlt_2^p$.  Since $f^*$ is its own inverse, we see that it is also a Borel reduction from $\qlt_2^p$ to $\qlt_2^s$.
\end{proof}

One can view $M_2$ as the complete binary tree $2^{<\N}$, with each word in $M_2$ corresponding to a node in the tree.  From this point of view, when looking at $A\subseteq M_2$, we see that $A_m$ is simply the set of words in $A$ which are above the node corresponding to $m$.  (See figure \ref{fig:trees}.)  This natural interpretation of one of the sets involved in $\qlt_2^p$ in terms of trees leads us to consider the quasi-order $\qlt_X^{tree}$.  Recall that for a countable discrete space $X$, a tree on $X$ is a (non-empty) collection of finite sequences of elements of $X$ which is closed under initial segments.  Let $\Lambda(X)$ be the Borel set of infinite trees in $Tr(X)$, the Polish space of trees on $X$.

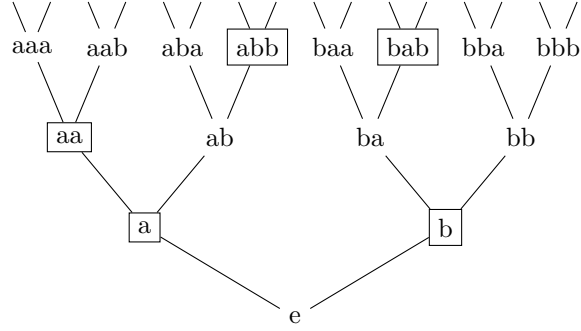
\begin{figure}
\begin{center}
\begin{tikzpicture}
[level distance=12 mm,
 level 1/.style={sibling distance=4cm},
 level 2/.style={sibling distance=2cm},
 level 3/.style={sibling distance=1cm},
 level 4/.style={sibling distance=5mm, level distance=6mm}]
\node{e}[grow'=up]
	child{node[rectangle,draw]{a}
		child{node[rectangle,draw]{aa}
			child{node{aaa} child child}
			child{node{aab} child child}
			}
		child{node{ab}
			child{node{aba} child child}
			child{node[rectangle,draw]{abb} child child}
			}
		}	
	child{node[rectangle,draw]{b}
		child{node{ba}
			child{node{baa} child child}
			child{node[rectangle,draw]{bab} child child}
			}
		child{node{bb}
			child{node{bba} child child}
			child{node{bbb} child child}
			}
		};	
\end{tikzpicture}
\parbox{5in}{\caption{The set $A=\{a,b,aa,abb,bab,\ldots\}$ in the binary tree corresponding to $M_2$.
 Note that, for example, $A_{ba}=\{bab,\ldots\}$ is the set of words in $A$ above $ba$.} \label{fig:trees} }
\end{center}
\end{figure}

Note that if we have $A,B\in \mc P(M_2)$ and $m\in M_2$ such that $mA=B_m$, and furthermore $A,B$ are both trees on $\{a,b\}$, then $m$ witnesses that $A\qlt_{\{a,b\}}^{tree} B$.  If $A$ or $B$ is not a tree, then it does not make sense to compare them using $\qlt_{\{a,b\}}^{tree}$, but this is only a minor difficulty, as we will see in the next proof.

\begin{theorem}\label{thm:TreesUniversal}
$\qlt_2^p \, \leq_B \, \qlt_2^{tree}\restriction\Lambda(2)$. It follows that $\qlt_2^{tree}$ is universal.
\end{theorem}
\begin{proof}
We will create our Borel reduction in two parts.  First we will define a Borel reduction from $\qlt_2^p$ to $\qlt_3^{tree}\restriction\Lambda(3)$.  Given $A\in\mc P(M_2)$, we define the tree $T_A\in Tr(3)$ as follows.  We start with the complete binary tree $2^{<\N}$, and add to it the sequence $\hat{w}^\frown 2$ iff $w\in A$, where $\hat{w}$ is the sequence in $2^{<\N}$ corresponding to the word $w$ in $M_2$.  This set is closed under initial segments and so is a tree.  Clearly it is infinite.  Let $T_A$ be this collection of sequences.

Suppose that $A\qlt_2^p B$.  Then there exists $m\in M_2$ such that $mA=B_m$.  First note that $2^{<\N}$ is contained in both $T_A$ and $(T_B)_{\hat{m}}$.  Next suppose that $w\in M_2$.  Then
\begin{align*}
\hat{w}^\frown 2\in T_A &\LongIff w\in A \\
 &\LongIff m^\frown w\in B \\
 &\LongIff \widehat{m^\frown w}^\frown 2=\hat{m}^\frown \hat{w}^\frown 2\in T_B \\
 &\LongIff \hat{w}^\frown 2\in (T_B)_{\hat{m}}
\end{align*}
So $T_A=(T_B)_{\hat{m}}$.

Conversely, suppose that $T_A=(T_B)_\alpha$ for some $\alpha\in 3^{<\omega}$.  If $\alpha$ contains a $2$, then $(T_B)_\alpha$ is $\{\emptyset\}$ or $\emptyset$, since the only sequences in $T_B$ containing $2$ are leaves of the tree.  However, $T_A$ is infinite.  So $\alpha\in 2^{<\omega}$, which means that there is a word $w\in M_2$ such that $\hat{w}=\alpha$.  Now
\begin{align*}
wx\in B &\LongIff \widehat{wx}^\frown 2\in T_B \\
 &\LongIff \hat{x}^\frown 2\in (T_B)_{\hat{w}} \\
 &\LongIff \hat{x}^\frown 2\in T_A \\
 &\LongIff x\in A,
\end{align*}
so $wA=B_w$.  Thus the map $t\from \mc P(M_2)\to Tr(3)$ sending $A$ to $T_A$ is a Borel reduction.

Next, we define a map $C\from Tr(3)\to Tr(2)$ which when composed with $t$ will be a Borel reduction from $\qlt_2^p$ to $\qlt_2^{tree}\restriction\Lambda(2)$.  First we inductively define a map $c\from 3^{<\N}\to 2^{<\N}$.  Let $c(e)=e$, $c(0)=00$, $c(1)=01$, and $c(2)=10$.  Now assume that $c$ has been defined for all words of length $\leq n$, and let $w=x^\frown u$, where $x\in\{a,b,c\}$ and $u\in 3^{<\N}$ has length $n$.  Define $c(w)=c(x)^\frown c(u)$.  Given $T\in Tr(3)$, apply $c$ to the elements of $T$ and close the resulting set under initial segments to get a tree $C(T)\in Tr(2)$.

Suppose that $t(A)\qlt_3^{tree} t(B)$, so there exists $u\in 3^{<\N}$ (in fact, $u\in 2^{<\N}$) such that $t(A)=t(B)_u$.  Then for all $w\in 3^{<\N}$
\begin{align*}
c(w)\in c(t(A)) &\LongIff w\in t(A) \\
 &\LongIff u^\frown w\in t(B) \\
 &\LongIff c(u^\frown w)=c(u)^\frown c(w)\in c(t(B))
\end{align*}
Hence $C(t(A))=C(t(B))_{c(u)}$ and thus $C(t(A))\qlt_2^{tree} C(t(B))$.

Now suppose that $C(t(A))\qlt_2^{tree} C(t(B))$, and so there exists $w\in 2^{<\N}$ such that $C(t(A))=C(t(B))_w$.  Suppose that $w$ is not in the image of $c$.  Then we either have $C(t(B))_w=\emptyset$, which is impossible, or $w$ is an initial segment of odd length of something in the image of $c$.  If $w$ ends in a $0$, then $100\in C(t(B))_w$, but this is not in $C(t(A))$.  If $w$ ends in a $1$, then $00\notin C(t(B))_w$, but $00\in C(t(A))$.  Thus $w$ is in the image of $c$, say $w=c(u)$.  Then
\begin{align*}
u^\frown v\in t(B) &\LongIff c(u)^\frown c(v)\in C(t(B)) \\
 &\LongIff c(v)\in C(t(A)) \\
 &\LongIff v\in t(A)
\end{align*}
Thus $t(A)=t(B)_u$, and so $t(A)\qlt_3^{tree} t(B)$.
\end{proof}

\section{Universal quasi-orders from group theory}\label{sec:GroupUniversal}

We have seen that $E_\infty$ is the same as the quasi-order $\qlt_{\F_2}^2$, and so our universal quasi-order $\qlt_{M_2}^2$ is a natural generalization of $E_\infty$.  At this point, we will turn our attention to other quasi-orders which can be seen as generalizations of $E_\infty$.  The most obvious generalization is the quasi-order $\subseteq^{\F_2,t}_{\mc P(\F_2)}$ on $\mc P(\F_2)$ defined by
$$ A\subseteq^{\F_2,t}_{\mc P(\F_2)} B \LongIff (\exists g\in\F_2)\hspace{4pt} gA\subseteq B. $$
Replacing the $\subseteq$ symbol on the right-hand side of the definition with the $=$ symbol gives $E_\infty$.  Unfortunately for our purposes, the above quasi-order is clearly not countable, and in fact has been shown to be a universal $K_\sigma$ quasi-order (see Louveau-Rosendal \cite{LR05}).  Consequently, $\subseteq^{\F_2,t}_{\mc P(\F_2)}$ is much more complex than any countable Borel quasi-order.  So we instead consider the following (countable Borel) quasi-order.

\begin{definition}\label{def:shiftisectqo}
If $G$ is a countable group, then $\qlt_t^G$ is the countable Borel quasi-order on $\mc P(G)$ defined by $$ A\qlt_t^G B  \LongIff (\exists g_1,\ldots,g_n\in G)\hspace{4pt} A=g_1 B\cap\ldots\cap g_n B.$$
\end{definition}

For any group $G$, let $\Omega(G)$ be the set of infinite subsets of $G$.  In order to show that $\qlt_t^{\F_2}$ is a universal countable Borel quasi-order, we will reduce $\qlt_2^{tree}\restriction\Lambda(2)$ to $\qlt_t^{\F_2}\restriction\Omega(\F_2)$.

Every tree on 2 is isomorphic to a tree $T$ on $\{a,b\}$, and these can easily be identified with subsets of $\F_2=\<a,b\>$.  If we take a subset $T\subset\F_2$ corresponding to a tree and multiply it on the left by $w^{-1}$, then the positive words in $w^{-1}T$ are precisely $T_w$.  Unfortunately, there is no natural way to pick out the positive words from $w^{-1}T$ simply by intersecting it with other shifts of $T$, and so we instead will define a set based on $T$ for which $T_w$ is easy to find simply by intersecting its shifts. In order to do this, we will look at subsets of $\F_\infty$, the free group on countably many generators.  We list the generators of $\F_\infty$ as
$$ \{ a,b,x_a,x_b,x_{aa},x_{ab},x_{ba},x_{bb},x_{aaa},\ldots \}. $$

Using the two generators $a,b$ we identify $T$ with a subset of the group, to which we add the sets $x_w wT_w$ for $w\in T$.  Call this new set $T'$.  Note that for all $w\in T$, $w^\frown T_w$ is a subset of $T$, and so $x_w wT_w\subseteq T'$.  Then $T'\cap x_w^{-1}T'=wT_w$, since $w T_w$ is the set of positive words in $x_w^{-1} T'$.  We can then multiply by $wT_w$ by $w^{-1}$ to find $T_w$.  However, the map sending $T$ to $T'$ is not a Borel reduction.  Although we can now find $T_w$ by intersecting shifts of $T'$, $T_w$ maps to $(T_w)'$, so that is the set we need to find.  The following proof addresses this issue.

\begin{theorem}\label{thm:IsectUniversal}
$\qlt_t^{\F_\infty}\restriction\Omega(\F_\infty)$ is a universal countable Borel quasi-order.
\end{theorem}
\begin{proof}
We will construct the reduction in a few steps.  We start with trees on $\{a,b\}$, which we then map to trees on $\{a,b,c,d\}$ for technical reasons.  Next we define a map $f\from \{a,b,c,d\}^{<\N}\to\mc P(\F_\infty)$, which will induce a map $F\from Tr(\{a,b,c,d\})\to\mc P(\F_\infty)$.  The composition of these two maps will be our reduction.

If $T\in Tr(\{a,b\})$, define
$$t_a(T)=\{w\in T \mid w^\frown a\notin T\}. $$
Similarly define $t_b(T)$.  These sets are elements of $T$ which are ``along the edge" of the tree, i.e.\ some immediate extension of these words is not in the tree.  We define $S\from Tr(\{a,b\})\to Tr(\{a,b,c,d\})$ by
\begin{equation}\label{Sdefinition}
S(T) = T \cup (t_a(T)^\frown c) \cup (t_b(T)^\frown d)
\end{equation}
where $X^\frown z=\{x^\frown z \mid x\in X\}$.  Here $S$ ``outlines" the tree using the letters $c$ and $d$.  The following property of $S$ will be important later.

\begin{lemma}\label{lemma:OutlineTree}
If $T,T'\in Tr(\{a,b\})$ and $S(T)\subseteq S(T')$, then $T=T'$.
\end{lemma}
\begin{proof}
It is easily seen that $S(T)\subseteq S(T')$ implies $T\subseteq T'$, as
$$S(T)\cap \{a,b\}^{<\N}=T \text{ and } S(T')\cap\{a,b\}^{<\N}= T'. $$
Suppose $w\in \{a,b\}^{<\N}\setminus T$.  Then there is some initial segment of ${w'}\subset w$ (possibly the empty string) and some $x\in\{a,b\}$ such that $w'\in t_x(T)$, i.e.\ $w={w'}^\frown x^\frown t$, where $w'\in T$, ${w'}^\frown x\notin T$, and $t\in\{a,b\}^{<\N}$.  Then ${w'}^\frown y\in S(T)$ for some $y\in\{c,d\}$, and so ${w'}^\frown y\in S(T')$.  This is only possible if ${w'}^\frown x$ and all its extensions are not in $T'$, and in particular $w\notin T'$.
\end{proof}

We list the generators of $\F_\infty$ as
$$\{a,b,c,d,x_a,x_b,x_c,x_d,x_{aa},x_{ab},x_{ac}\ldots\}$$
i.e.\ every string in $\{a,b,c,d\}^{<\N}$ (except the empty string) has a unique generator associated to it in addition to generators corresponding to the letters in our trees.  The empty string in $\{a,b,c,d\}^{<\N}$ and the identity element in $\F_\infty$ will both be written as $e$.  This should not cause confusion, although both uses will appear close to each other.  Finally, we recall that if $A,B\in\mc P(\F_\infty)$, then \mbox{$AB=\{ab\mid a\in A, b\in B\}$}.  We can now define $f\from\{a,b,c,d\}^{<\N}\to\mc P(\F_\infty)$ inductively.

\begin{center}
\begin{align*}
f(e) &= \{e\} \\
f(a) &= \{a,x_a a\} \\
f(b) &= \{b,x_b b\} \\
f(c) &= \{c,x_c c\} \\
f(d) &= \{d,x_d d\} \\
f(w) &= \left( \bigcup_{\substack{w=s^\frown t \\ s,t\neq e}} f(s)f(t) \right) \cup \{x_w w\}
\end{align*}
\end{center}

The idea here is that every set $f(w)$ contains elements which encode the relation of $w$ to its initial segments.  Then define $F\from Tr(\{a,b,c,d\})\to \mc P(\F_\infty)$ by
$$F(T)=\bigcup_{w\in T} f(w). $$

There are a few helpful facts to record at this point.  The simplest one is that $w\in f(w)$, which follows by a simple induction.  The others we record as lemmas.

\begin{lemma}
If $u,v\in \{a,b,c,d\}^{<\N}$ are not equal, the sets $f(u)$ and $f(v)$ are disjoint.
\end{lemma}
\begin{proof}
Define the function $\Phi\from \F_\infty\to \{a,b,c,d\}^{<\N}$ as
\begin{align*}
\Phi(g) = &\text{the word in $\{a,b,c,d\}^{<\N}$ obtained by removing all other letters}\\
 & \text{from the freely reduced representation of $g$.}
\end{align*}
By a simple inductive argument we see that for all $w\in \{a,b,c,d\}^{<\N}$, $\Phi$ is constant on $f(w)$ and equal to $w$.  Thus the sets are disjoint.
\end{proof}

\begin{lemma}\label{uextendsw}
If a word starting with $x_w$ is in $f(u)$, then $w\subset u$.
\end{lemma}
\begin{proof} This follows from an straightforward induction on the length of $u$. \end{proof}

\begin{lemma}\label{wordinfwft}
If $\gamma\in f(u)$ starts with $x_w w$ and $u=w^\frown t$, then $\gamma=x_w w \lambda$, with $\lambda\in f(t)$.
\end{lemma}
\begin{proof}
If $t=e$, then $\gamma=x_w w$.  Otherwise, there must be some $\alpha,\beta$ such that $u=\alpha^\frown \beta$ and $\gamma\in f(\alpha)f(\beta)$.  We can then split $\gamma$ into two words, $\gamma=\delta\lambda$, where $\delta$ starts with $x_w w$ and $\delta\in f(\alpha)$, while $\lambda\in f(\beta)$.  By the previous lemma, $w\subset\alpha$, say $\alpha=w^\frown z$.  Then $u=w^\frown z^\frown \beta$. We write $t=z^\frown \beta$.  By induction, $\delta=x_w w \delta'$ with $\delta'\in f(z)$.  Then $\gamma=x_w w \delta'\lambda$, and $\delta'\lambda\in f(z)f(\beta)\subseteq f(t)$ by definition.
\end{proof}

We define the map \mbox{$G\from Tr(\{a,b\})\to \mc P(\F_\infty)$} by
$$ G(T) = F(S(T)) $$
where $S$ is the map defined in \eqref{Sdefinition}.

\begin{lemma}\label{itsahom}
For all $w\in \{a,b\}^{<\N}$ and all nonempty $T\in Tr(\{a,b,c,d\})$,
$$G(T)\cap x_w^{-1} G(T)= wG(T_w)$$
and hence
$$w^{-1}G(T)\cap (x_w w)^{-1} G(T)=G(T_w). $$
\end{lemma}
\begin{proof}
First, we will show that $G(T)\cap x_w^{-1} G(T)\subseteq wG(T_w)$.  Every element of $G(T)$ is a positive word in the generators of $\F_\infty$, so any word not starting with $x_w$ will be freely reduced in $x_w^{-1} G(T)$ and so begin with $x_w^{-1}$, and thus not be in $G(T)$.  So we need only focus on the words that start with $x_w$.

Suppose $g\in f(u)\subseteq G(T)$ and $g=x_w \alpha$ for some $\alpha\in \F_\infty$.  By our inductive definition, this implies $g=x_w w \beta$ for some $\beta\in \F_\infty$.  By Lemma \ref{uextendsw}, we must have $u=w^\frown t$ for some $t\in \{a,b,c,d\}^{<\N}$.  By Lemma \ref{wordinfwft}, $\beta\in f(t)$.  Also, $w\beta$ is in $f(w)f(t)$, so $w\beta\in G(T)\cap x_w^{-1}G(T)$.  In addition, $\beta\in G(T_w)$, since $t\in (S(T))_w=S(T_w)$ (since $w\in\{a,b\}^{<\N}$) and so $f(t)\subseteq G(T_w)$.  Thus $G(T)\cap x_w^{-1} G(T)\subseteq wG(T_w)$.

If $g\in G(T_w)$, then there is some $u\in S(T_w)$ such that $g\in f(u)$.  Then
$$x_w wg,wg\in f(w)f(u)\subseteq G(T)$$
so $wg \in G(T)\cap x_w^{-1} G(T)$.  Thus $G(T)\cap x_w^{-1} G(T)\supseteq wG(T_w)$.
\end{proof}

Lemma \ref{itsahom} shows that for $T,S\in\Lambda(2)$, if $T\qlt_2^{tree} S$, then $G(T)\qlt_t^{\F_\infty} G(S)$.  Next we check the other direction.

Suppose that $G(T)=g_1 G(T') \cap \ldots \cap g_n G(T')$.  We know that $e\in G(T)$ (since $T$ is nonempty), which means that each $g_i$ must be an inverse of an element in $G(T')$, say $g_i^{-1}=h_i\in G(T')$.  Fix some $1\leq i\leq n$ and suppose that $\Phi(h_i)=w$, i.e.\ $h_i\in f(w)$.  If $u\in S(T)$, then $x_u u\in G(T)$.  This implies $h_i x_u u \in G(T')\cap f(w^\frown u)$, and in particular the intersection is nonempty, so $w^\frown u \in S(T')$. Thus $S(T)\subseteq S(T')_w$.

If $w\notin\{a,b\}^{<\N}$, then $S(T')_w$ is either empty or a single element, but $S(T)$ is infinite.  Thus $w\in\{a,b\}^{<\N}$, and so $S(T')_w=S(T'_w)$.  It follows that $S(T)\subseteq S(T'_w)$, and so by Lemma \ref{lemma:OutlineTree}, $T=T'_w$.  Thus $G$ is a Borel reduction.  This completes the proof of Theorem \ref{thm:IsectUniversal}.
\end{proof}

\begin{cor}\label{cor:TranslateUniversal}
$\qlt_t^{\F_2}\restriction\Omega(\F_2)$ is a universal countable Borel quasi-order, and so $\qlt_t^{\F_2}$ is a universal countable Borel quasi-order.
\end{cor}
\begin{proof}
Let $\phi\from\F_\infty\to\F_2$ be an embedding.  Then $\phi$ induces a map
\begin{align*}
\Phi\from\Omega(\F_\infty) &\to\Omega(\F_2) \\
A &\mapsto \{ \phi(a) \mid a\in A \}
\end{align*}
If $A,B\in\Omega(\F_\infty)$ and there exist $g_1,\ldots,g_n\in\F_\infty$ such that
$$A=g_1 B\cap\ldots\cap g_n B$$
then $\Phi(A)=\phi(g_1)\Phi(B)\cap\ldots\cap \phi(g_n)\Phi(B)$.

Conversely, suppose that
$$\Phi(A)=h_1\Phi(B)\cap\ldots\cap h_n\Phi(B). \hspace{.5in} (*)$$
If some $h_i$ is not in the image of $\phi$, then $h_i\Phi(B)$ is disjoint from any set in the image of $\Phi$, and so the right hand side cannot equal the left hand side unless $\Phi(A)=\emptyset$, which is impossible.  This implies that every $h_i$ in $\displaystyle{(*)}$ is in the image of $\phi$.  It follows that $A=\phi^{-1}(h_1) B\cap\ldots\cap\phi^{-1}(h_n) B$.
\end{proof}

\begin{remark}
The above proof shows that if $G$ is any countable group containing $\F_2$ as a subgroup, then $\qlt_t^G\restriction\Omega(G)$ is a universal countable Borel quasi-order.
\end{remark}

Let $E_c(G)$ denote the conjugacy equivalence relation on the standard Borel space $\Sg(G)$ of subgroups of $G$, i.e.\ for $A,B\in\Sg(G)$,
$$ A \mathbin{E_c(G)} B \LongIff (\exists g\in G)\, A=gBg^{-1}. $$
In \cite{Gao00}, Gao used a simple coding technique to prove the following result.

\begin{theorem}[Gao]
If $G=K*H$, where $K$ has a nonabelian free subgroup and $H$ is nontrivial cyclic, then $E_c(G)$ is a universal countable Borel equivalence relation.
\end{theorem}

In light of the relationship between $E_\infty$ and $\qlt_t^{\F_2}$, it is natural to consider the following countable Borel quasi-order:

\begin{definition}\label{def:subgpisectqo}
Let $G$ be a countable group.  Then $\qlt_c^G$ is the countable Borel quasi-order on $\Sg(G)$ defined by
$$A\qlt_c^G B \LongIff (\exists g_1,\ldots,g_n\in G)\hspace{4pt} A=g_1 B g_1^{-1}\cap\ldots\cap g_n B g_n^{-1}.$$
\end{definition}

Let $\Gamma(G)$ be the standard Borel space of infinite subgroups of $G$.  Then the proof of the following result is a straightforward adaptation of Gao's argument in \cite{Gao00}.

\begin{theorem}\label{thm:ConjUniversal}
Suppose that $G$ is a countable group containing a nonabelian free subgroup and that $H$ is a nontrivial cyclic group.  Then $\qlt_c^{G*H}\restriction\Gamma(G*H)$ is a universal countable Borel quasi-order, and so $\qlt_c^{G*H}$ is a universal countable Borel quasi-order.
\end{theorem}
\begin{proof}
Let $h\in H$ be a generator of $H$.  We define the map $K\from\Omega(G)\to\Sg(G*H)$ by
$$ K(A)=\<xhx^{-1} \colon x\in A\>. $$
This map is Borel, so we need only check that it is a reduction from $\qlt_t^{G}\restriction\Omega(G)$ to $\qlt_c^{G*H}\restriction\Gamma(G*H)$.  We will make use of the observation that $K(A)=\underset{g\in A}{*}\, gHg^{-1}$.

If $A,B\in\mc P(G)$, then clearly $K(A\cap B)\subseteq K(A)\cap K(B)$.  We will show that $K(A)\cap K(B)\subseteq K(A\cap B)$, so $K(A\nolinebreak\cap\nolinebreak B)\nolinebreak=\nolinebreak K(A)\cap K(B)$.  Suppose that \mbox{$g\in K(A)\cap K(B)$}, and so can be written both as $g=x_1 h x_1^{-1}\ldots x_n h x_n^{-1}$ with $x_1,\ldots,x_n\in A$ and as $g=y_1 h y_1^{-1}\ldots y_m h y_m^{-1}$ with $y_1,\ldots,y_m\in B$.  Then clearly $x_1=y_1$, and so multiplying $g$ on the left by $y_1 h^{-1} y_1^{-1}=x_1 h^{-1} x_1^{-1}$ we find that
$$x_2 h x_2^{-1}\ldots x_n h x_n^{-1}=y_2 h y_2^{-1}\ldots y_m h y_m^{-1}. $$
Thus $x_2=y_2$, and repeating this argument we find that $x_i=y_i$ for \mbox{$1\leq i\leq\min\{n,m\}$}.  If for example $m<n$, then we would have the equation
$$x_{m+1} h x_{m+1}^{-1}\ldots x_n h x_n^{-1}=e$$
which is impossible.  Similarly it cannot be that $n<m$.  Thus $m=n$, and it follows that $g\in K(A\cap B)$.

Also note that if $g\in G$, then $K(gA)=g K(A) g^{-1}$.  Thus
\begin{align*}
K(g_1 A\cap\ldots\cap g_n A) &= K(g_1 A)\cap\ldots\cap K(g_n A) \\
 &=g_1 K(A) g_1^{-1} \cap \ldots \cap g_n K(A) g_n^{-1}
\end{align*}

Suppose that $A,B\in\mc P(G)$ and that $A\qlt_t^G B$, i.e.\ there exist $g_1,\ldots,g_n\in G$ such that $A=g_1 B\cap\ldots\cap g_n B$.  Then
$$K(A)=K(g_1 B \cap \ldots \cap g_n B)=g_1 K(B) g_1^{-1} \cap \ldots \cap g_n K(B) g_n^{-1}. $$
Thus $A\qlt_t^G B$ implies that $K(A) \qlt_c^{G*H} K(B)$.

Next, suppose that $A,B\in\mc P(G)$ and that $K(A)\qlt_c^{G*H} K(B)$, so there exist $\gamma_1,\ldots,\gamma_n\in G*H$ such that $K(A)=\gamma_1 K(B) \gamma_1^{-1} \cap \ldots \cap \gamma_n K(B) \gamma_n^{-1}$.  For each $x\in A$ and $1\leq i\leq n$, let $w_{x,i}\in K(B)$ be the element such that $xhx^{-1}=\gamma_i w_{x,i} \gamma_i^{-1}$.  Clearly for each $i=1,\ldots,n$ the map $x\mapsto w_{x,i}$ is an injection.

Note that $xhx^{-1}$ is a reduced word in $G*H$.  For $1\leq i\leq n$, we may assume that $\gamma_i$ is a reduced word in $G*H$, and that $w_{x,i}\in K(B)$ can be written as
$$w_{x,i} = z_1 h^{\epsilon_1} z_1^{-1}\ldots z_k h^{\epsilon_k} z_k^{-1} \,\,(z_j\in B, \epsilon_j\in\{\pm 1\}). $$
If we reduce this word, then we obtain that
$$w_{x,i}=u_1 h^{m_1} u_2 h^{m_2}\ldots u_t h^{m_t} u_{t+1}$$
where $m_j\in\Z\setminus\{0\}$, $u_j\in G$ and the product $u_1u_2\ldots u_j\in B$ for $1\leq j\leq t+1$.  Furthermore, $w_{x,i}$ is never the trivial word.

The equation $xhx^{-1}=\gamma_i w_{x,i} \gamma_i^{-1}$ implies that starting with the right-hand side, there is a cancellation procedure which eventually leads to the left-hand side.  In any such procedure, there must be some occurrence of $h$ in the right-hand side which is never cancelled.  We call this the preserved occurrence of $h$.  Let $\Delta_i\subseteq A$ be the set of elements $x\in A$ for which the preserved occurrence of $h$ in some cancellation procedure is in the original expression for $w_{x,i}$.

We claim that $A\setminus\Delta_i$ is finite for each $1\leq i\leq n$.  If $x\in A\setminus\Delta_i$, then the preserved occurrence of $h$ is in either $\gamma_i$ or $\gamma_i^{-1}$.  Suppose that $x_1,x_2\in A\setminus\Delta_i$ are both words such that the preserved occurrence of $h$ is in $\gamma_i$.  Then the preserved occurrence of $h$ must be the first $h$ in $\gamma_i$, since $\gamma_i$ is assumed to be reduced.  Thus $\gamma_i=k h u$ for some $k\in G,u\in G*H$, and this gives us the two equations
\begin{align*}
x_1 h x_1^{-1} = khuw_{x_1,i}\gamma_i^{-1} \\
x_2 h x_2^{-1} = khuw_{x_2,i}\gamma_i^{-1}
\end{align*}
which implies that $x_1=k=x_2$.  Thus there is at most one element in $A\setminus\Delta_i$ such that the preserved occurrence of $h$ is in $\gamma_i$.  A similar argument shows that there is at most one element in $A\setminus\Delta_i$ such that the preserved occurrence of $h$ is in $\gamma_i^{-1}$.  So $|A\setminus\Delta_i|\leq 2$ for $1\leq i\leq n$.

As $A$ is infinite, this implies that $\Delta=\cap_{1\leq i\leq n} \Delta_i$ is infinite.  If we fix some $x_0\in\Delta$, then for each $i=1,\ldots,n$, we can write
\begin{align*}
x_0 hx_0^{-1} &=\gamma_i w_{x_0,i} \gamma_i^{-1} \\
 &=\gamma_i u_i (z_i h z_i^{-1}) v_i \gamma_i^{-1}
\end{align*}
with $z_i\in B$, $u_i,v_i\in K(B)$, and the displayed $h$ is the preserved occurrence in some cancellation procedure.  This implies that $x_0 = \gamma_i u_i z_i$, and $x_0^{-1}=z_i^{-1} v_i \gamma_i^{-1}$.  Let $\beta_i=x_0 z_i^{-1}\in G$.  Then $\gamma_i=\beta_i u_i^{-1}$.  Thus
\begin{align*}
K(A) &= \beta_1 u_1^{-1} K(B) u_1 \beta_1^{-1} \cap \ldots \cap \beta_n u_n^{-1} K(B) u_n \beta_n^{-1} \\
 &= \beta_1 K(B) \beta_1^{-1} \cap \ldots \cap \beta_n K(B) \beta_n^{-1} \\
 &= K(\beta_1 B \cap \ldots \cap \beta_n B)
\end{align*}
and so $A=\beta_1 B \cap \ldots \cap \beta_n B$, with each $\beta_i\in G$.  Thus $A\qlt_t^G B$, as desired.
\end{proof}

The following result is an immediate consequence of Theorem \ref{thm:ConjUniversal}.

\begin{cor}
If $n\geq 3$, then $\qlt_c^{\F_n}\restriction\Gamma(\F_n)$ is a universal countable Borel quasi-order.
\end{cor}

Finally, the proof of the following result is a straightforward adaptation of the proof of Proposition 1 of Thomas-Velickovic \cite{TV99}.

\begin{cor}
$\qlt_c^{\F_2}\restriction\Gamma(\F_2)$ is a universal countable Borel quasi-order.
\end{cor}
\begin{proof}
Recall that a subgroup $H$ of a group $G$ is said to be malnormal if $gHg^{-1}\cap\nolinebreak H=\nolinebreak\{1\}$ for all $g\in G\setminus H$, and that $\F_3$ can be embedded as a malnormal subgroup of $\F_2$.  Arguing as in Corollary \ref{cor:TranslateUniversal}, we see this embedding induces a Borel reduction from $\qlt_c^{\F_3}\restriction\Gamma(\F_3)$ to $\qlt_c^{\F_2}\restriction\Gamma(\F_2)$.
\end{proof}

\section{Embeddability of countable groups}\label{sec:CtblGroupEmbed}

Our ultimate goal is to show that embeddability of finitely generated groups is a universal countable Borel quasi-order.  The techniques we will use in the proof are easier to understand in the more general setting of arbitrary countable groups.  With this in mind, we first turn our attention to the embeddability relation for countable groups, $\sqsubseteq_{Gp}$.  By removing the restriction that the groups we work with should be finitely generated, we are allowed more freedom with regards to how we construct groups for our Borel reduction.  At the same time, removing this restriction means that $\sqsubseteq_{Gp}$ is an analytic quasi-order, rather than a countable Borel quasi-order.  We will prove the following result.

\begin{theorem}\label{thm:EmGroupsUniversal}
$\sqsubseteq_{Gp}$ is a universal analytic quasi-order.
\end{theorem}

\begin{cor}
The bi-embeddability relation for countable groups $\equiv_{Gp}$ is a universal analytic equivalence relation.
\end{cor}

This is in contrast with the isomorphism relation for countable groups $\cong_{Gp}$, which is known to be universal among all analytic equivalence relations induced by a Borel action of $S_\infty$.  (This is due to Mekler in \cite{Mek81}.)  However, such equivalence relations are known not to be universal among all analytic equivalence relations.

Before we prove Theorem \ref{thm:EmGroupsUniversal}, we need to make a few definitions.  We will write $\mc C$ for the set of countable graphs whose vertex set is $\N$.  By identifying each graph with its edge relation, we see that $\mc C$ is a closed subset of $2^{\N^2}$ and so is a Polish space.

\begin{definition}
If $S,T\in \mc C$, then we write $S \sqsubseteq_{\mc C} T$ if $S$ embeds into $T$, i.e.\ there exists $f\from \N\to \N$ such that for all $m,n\in\N$, $(m,n)\in S \Iff (f(m),f(n))\in T$.
\end{definition}

In \cite{LR05}, it was shown that $\sqsubseteq_{\mc C}$ is a universal analytic quasi-order.  Thus to show that $\sqsubseteq_{Gp}$ is universal, we need only show that $\sqsubseteq_{\mc C}$ Borel reduces to it.  To do this, we will use small cancellation techniques to create groups that encode the edge relations of graphs.  We recall the following definitions and theorems from small cancellation theory.  (See Chapter V of \cite{LS01} for a fuller treatment of small cancellation theory.)

\begin{definition}
Let $R$ be a subset of a free group $F$.  We say $R$ is \textit{symmetrized} if every element of $R$ is cyclically reduced and whenever $r\in R$, all cyclic permutations of $r$ and $r^{-1}$ are in $R$.
\end{definition}

\begin{theorem}\label{thm:wordproblem}[Theorem V.4.4 in \cite{LS01}]
Let $F$ be a free group.  Let $R$ be a symmetrized subset of $F$ and $N$ its normal closure.  If $R$ satisfies $C'(\lambda)$ for some $\lambda\leq 1/6$, then every non-trivial element $w\in N$ contains a subword $s$ of some $r\in R$ with $|s|>(1-3\lambda)|r|\geq \frac{1}{2}|r|$.
\end{theorem}

\begin{theorem}\label{thm:torsion}[Theorem V.10.1 in \cite{LS01}]
Suppose that $G=\<x_1,x_2,\ldots \mid R\>$ is such that $R$ is a symmetrized subset of $\<x_1,x_2,\ldots\>$ satisfying the $C'(1/6)$ small cancellation condition.  If $w$ represents a word of finite order in $G$, then there is some $r\in R$ of the form $r=v^n$ such that $w$ is conjugate to a power of $v$.
\end{theorem}

Note that this implies that if furthermore $w$ is cyclically reduced, then $w$ is a cyclic permutation of some $r\in R$ of the form $r=v^n$.  We will often refer to this consequence of Theorem \ref{thm:torsion}.

Now we can proceed to the proof of Theorem \ref{thm:EmGroupsUniversal}.

\begin{proof}[Proof of Theorem \ref{thm:EmGroupsUniversal}]

Let $T\in\mc C$ and let $v_0,v_1,\ldots$ be an enumeration of the vertices of $T$.  Then $G_T$ is defined to be the group with generators $v_0,v_1,\ldots$ and relators

\begin{itemize}
\item $v_i^7$ for all $i\in\N$
\item $(v_i v_j)^{11}$ if $(v_i,v_j)\in T$
\item $(v_i v_j)^{13}$ if $(v_i,v_j)\notin T$
\end{itemize}

Let $R_T$ be the symmetrization of the set of defining relations for $G_T$.  Note that if $T$ is any graph, then $R_T$ satisfies the $C'(1/6)$ condition.  Now suppose that $S,T\in\mc C$ are such that $S$ embeds into $T$, say via the map $f$.  Then $f$ extends to a group homomorphism from $G_S$ to $G_T$, as it sends the relations of $G_S$ to relations of $G_T$.

To see that $f$ is an embedding, let $\alpha=v_{i_1}^{k_1}v_{i_2}^{k_2}\ldots v_{i_n}^{k_n}$ be a word in the generators of $G_S$, so that $$f(\alpha)={f(v_{i_1})}^{k_1}{f(v_{i_2})}^{k_2}\ldots {f(v_{i_n})}^{k_n}$$
and suppose that $f(\alpha)=1$ in $G_T$.  Then by the $C'(1/6)$ condition, $f(\alpha)$ must contain more than $1/2$ of a relation in $R_T$.  Note that any such relation must involve only generators in the image of the graph embedding $f\from S\to T$.  Suppose that $f(\alpha)$ contains more than half of a relation of the form $f(v_i)^{\pm 7}$.  Since $f$ is one-to-one, this cannot happen unless $\alpha$ already contained more than half of $v_i^{\pm 7}$.

Suppose that $f(\alpha)$ contains more than $1/2$ of the relation $({f(v_i)} {f(v_j)})^k$, where the value of $k$ depends on whether or not $(f(v_i),f(v_j))\in T$.  Since
$$(v_i,v_j)\in S \Iff (f(v_i),f(v_j))\in T$$
it must be the case that $(v_i v_j)^k\in R_S$, and $\alpha$ already contained more than $1/2$ of $(v_i v_j)^k$.  Thus $f(\alpha)$ does not contain more than $1/2$ of a relation in $R_T$ unless $\alpha$ contains more than $1/2$ of the corresponding relation in $R_S$.  Since every nontrivial element in $G_S$ may be written as a word which does not contain more than $1/2$ of a relation in $R_S$, every nontrivial element in $G_S$ maps to a nontrivial element in $G_T$.  Thus if $S$ embeds into $T$, then $G_S$ embeds into $G_T$.

Conversely, suppose that $\theta\from G_S\to G_T$ is an embedding.  Let $v_0,v_1,\ldots$ enumerate the vertices of $S$.  By Theorem \ref{thm:torsion}, after adjusting the embedding $\theta$ by an inner automorphism of $G_T$ if necessary, we may assume $\theta(v_0)=t_0^k$ for some $k$ such that $|k|<7$, where $t_0$ is some vertex of $T$, since $\theta(v_0)$ must have order 7.  Let $v_j\neq v_0$ be some vertex of $S$.  Again by Theorem \ref{thm:torsion}, we find that $\theta(v_j)=u t_j^l u^{-1}$ for some $l$ such that $|l|<7$, where $u\in G_T$ and $t_j$ is some vertex of $T$.  Unfortunately we cannot eliminate $u$ by an inner automorphism without possibly changing the value of $\theta(v_0)$.  We may assume that $u$ is freely reduced and does not start with any power of $t_0$.  To see this, note that if $u$ began with $t_0^m$, then we would be able to follow $\theta$ by the inner automorphism corresponding to $t_0^m$ without changing the value of $\theta(v_0)$.  Thus $\theta(v_0 v_j)=t_0^k u t_j^l u^{-1}$ is cyclically reduced.  Since $v_0 v_j$ is a torsion element, so is $\theta(v_0 v_j)$.  By Theorem \ref{thm:torsion}, $\theta(v_0 v_j)$ must be a cyclic permutation of some $r\in R_T$.  It immediately follows that $u=1$, since no such words contain a mix of positive and negative powers.  Thus $\theta(v_0 v_j)=t_0^k t_j^l$.

From this we find that $t_0\neq t_j$, since otherwise $\theta(v_0 v_j)$ would have order 1 or 7, which is impossible since $\theta$ is an embedding and $v_0 v_j$ has order 11 or 13.  Again, by Theorem \ref{thm:torsion}, we find that $t_0^k t_j^l$ has finite order only if $k=l=\pm 1$.  As the orders of $v_0 v_j$ and $\theta(v_0 v_j)=t_0^{\pm 1} t_j^{\pm 1}$ are equal, we see that
$$(v_0,v_j)\in S \LongIff (t_0,t_j)\in T. $$

Let $v_m\neq v_n$ be arbitrary vertices in $S$.  Repeating the above argument with $v_0$ and $v_m$, as well as $v_0$ and $v_n$, we find there are inner automorphisms $\psi_1,\psi_2$ of $G_T$, corresponding to conjugating by suitable powers of $t_0$, such that $\psi_1(\theta(v_m))=t_m^{\pm 1}$ and $\psi_2(\theta(v_n))=t_n^{\pm 1}$, where $t_m\neq t_0$ and $t_n\neq t_0$.  A priori it may be the case that, for example, $\psi_1(\theta(v_n))=t_0^k t_n^{\pm 1} t_0^{-k}$, with $k\neq 0$. But then 
$$\psi_1(\theta(v_m v_n))=t_m^{\pm 1} t_0^k t_n^{\pm 1} t_0^{-k}$$
has infinite order, which is impossible.  Thus $\psi_1=\psi_2$, and so $\psi_1(\theta(v_m v_n))=t_m^{\pm 1} t_n^{\pm 1}$, and the above argument shows $t_m\neq t_n$ and that
$$(v_m,v_n)\in S \LongIff (t_m,t_n)\in T. $$
As $v_m$ and $v_n$ were arbitrary, the function $g\from S\to T$ defined by $g(v_i)=t_i$ for all $i\in\N$ is an embedding.  Thus $\sqsubseteq_{\mc G} \;\, \leq_B \;\, \sqsubseteq_{Gp}$, which establishes the result.
\end{proof}

\section{Embeddability of finitely generated groups}\label{sec:FGGroupEmbed}

We now turn our attention to the embeddability relation for finitely generated groups.

\begin{definition}
Let $\mc G$ denote the Polish space of finitely generated groups.  (See \cite{C00}.)  If ${\mbox A,B\in\mc G}$, then we write $A\qlt_{em} B$ if and only if there is a group embedding from $A$ into $B$.  We write $\equiv_{em}$ for the associated equivalence relation.
\end{definition}

It can be seen that $\qlt_{em}$ is a countable Borel quasi-order, as any finitely generated group contains only countably many finitely generated subgroups.  We will show that in fact it is universal by reducing $\qlt_2^{tree}$ to $\qlt_{em}$.  Given a tree $T\in Tr(2)$, our general strategy is to define a finitely generated group $G_T$ with subgroups corresponding to the trees $T_w$ for $w\in 2^{<\N}$.  We will start with two generators and then add relations to this group according to the nodes present in $T$.  As in the previous section, these additional relations will allow us to control the embeddings that exist between two of these groups and thus ensure that $T\mapsto G_T$ is a Borel reduction.  Thus we will have shown:

\begin{theorem}\label{thm:FGGroupEmbedUniversal}
$\qlt_{em}$ is a universal countable Borel quasi-order.
\end{theorem}

In order to define the relations of $G_T$, we first define the following two homomorphisms:

\begin{align*}
f_0 \from \F_2 &\to\F_2 & f_1 \from \F_2 &\to\F_2 \\ 
x &\mapsto x^5y & x &\mapsto x^2yxyx \\
y &\mapsto y^5x & y &\mapsto y^2xyxy
\end{align*}

We also define $f_e$ to be the identity map.  For any element $w\in 2^{<\N}$, we define $f_w$ to be the corresponding composition of $f_0$ and $f_1$, e.g.\ $f_{01}=f_0\circ f_1$ and $f_{110}=f_1\circ f_1\circ f_0$.  In other words, if we can write $w$ as $u^\frown v$, then $f_w=f_u\circ f_v$.  The associativity of function composition ensures that this is well-defined.

One basic property of these maps is that for all $u\in 2^{<\N}$, the first letter of $f_u(a)$ is different for each $a\in\{x^{\pm 1}, y^{\pm 1}\}$, and the same is true for the last letter.  This can be established through an easy induction on the length of $u$.  If $u=e$, then this is immediate, and for $u=0$ or $u=1$, we quickly check that it holds.  Now suppose that this is true for $u$.  Then for $i\in\{0,1\}$, consider $f_{u^\frown i}(a)=f_u(f_i(a))$.  We have already seen that the first and last letters of $f_i(a)$ are different for each $a$.  By assumption, $f_u$ takes the first and last letters of $f_i(a)$ to words with first and last letters different from those of $f_i(b)$ for any $b\neq a$ with $b\in\{x^{\pm 1},y^{\pm 1}\}$, and this completes the induction.

With this established, a similar induction shows that every $f_u$ takes freely reduced words to freely reduced words.  In fact, every $f_u$ takes cyclically reduced words to cyclically reduced words, since for $a,b\in\{x^{\pm 1}, y^{\pm 1}\}$, the first letter of $f_u(a)$ is the inverse of the last letter of $f_u(b)$ only if $b=a^{-1}$, by the uniqueness of the last letters.

If $\alpha\in\F_2$, then for any $w\in 2^{<\N}$, we can think of $f_w(\alpha)$ as a word on
$$\{f_w(a) \mid a\in\{x^{\pm 1}, y^{\pm 1}\}\}.$$
We refer to these special subwords as $f_w$-blocks.  See figure \ref{blocksbasic}.

\begin{figure}
$$\underbrace{ \underbrace{x^5yx^5yx^5yx^5yx^5yy^5x}_{f_{00}(x)} \underbrace{x^5yx^5yx^5yx^5yx^5yy^5x}_{f_{00}(x)} \underbrace{x^5yx^5yx^5yx^5yx^5yy^5x}_{f_{00}(x)}}_{3\,\mathrm{times}} $$
\caption{$f_{00}(x)^3$, with its $f_{00}$-blocks shown}\label{blocksbasic}
\end{figure}

Given $T\in Tr(2)$, we define 
$$G_T=\langle x,y \mid \{(f_w(x))^{59}, (f_w(y))^{61} \mid w\in T\}, \{(f_w(x))^{67}, (f_w(y))^{71} \mid w\notin T\}\rangle. $$
The numbers in the exponents were chosen to be relatively prime and so that the relations satisfy small cancellation conditions, and they have no significance beyond that.  We will eventually show that the map $T\mapsto G_T$ is a Borel reduction from $\qlt_2^{tree}$ to $\qlt_{em}$.  To see that this is the case, we proceed by a series of lemmas.

\begin{lemma}\label{lemma:lineup}
Let $u\in 2^{<\N}$, $a,b,c\in\{x^{\pm 1}, y^{\pm 1}\}$.  Suppose $f_u(a)$ is a subword of \mbox{$f_u(bc)=f_u(b)f_u(c)$}.  Then $f_u(a)$ does not contain letters from both $f_u(b)$ and $f_u(c)$.  In other words, $f_u(a)$ must equal $f_u(b)$ or $f_u(c)$.  It follows that if $\alpha,\beta\in\F_2$ are nontrivial and $f_u(\alpha)$ is a subword of $f_u(\beta)$, then $\alpha$ is a subword of $\beta$.
\end{lemma}
\begin{proof}
We prove this inductively.  It is easily checked to be true in the case that $u\in\{0,1\}$.  Now suppose that it is true for all $u$ with $|u|<n$.  Then if $u'=u^\frown i$ with $|u'|=n$ and $i\in\{0,1\}$ we may write $f_{u'}(a)=f_u(f_i(a))$ and  $f_{u'}(bc)=f_u(f_i(bc))$.  By assumption, the $f_u$-blocks in $f_{u'}(a)$ line up with the $f_u$-blocks in $f_{u'}(bc)$, and since $f_{u'}(a)$ is a subword of $f_{u'}(bc)$, it follows that $f_i(a)$ is a subword of $f_i(bc)$.  This implies that $f_i(a)$ equals $f_i(b)$ or $f_i(c)$.  Thus we find $f_{u'}(a)$ equals $f_{u'}(b)$ or $f_{u'}(c)$.
\end{proof}

\begin{lemma}\label{functioncyclic}
Let $w\in 2^{<\N}$, $\alpha,\beta\in\F_2$.  If $f_w(\alpha)$ is a cyclic permutation of $f_w(\beta)$, then $\alpha$ is a cyclic permutation of $\beta$.
\end{lemma}
\begin{proof}
We may write $\alpha=a_1\ldots a_n$ and $\beta=b_1\ldots b_n$ with $a_i,b_j\in\{x^{\pm 1},y^{\pm 1}\}$, $1\leq i,j\leq n$.  Thus $f_w(\alpha)=f_w(a_1)\ldots f_w(a_n)$ and $f_w(\beta)=f_w(b_1)\ldots f_w(b_n)$.  As $f_w(\alpha)$ is a cyclic permutation of $f_w(\beta)$, there is some $1\leq k\leq n$ such that if we write $\alpha=\alpha_1 a_k \alpha_2$, then
$$f_w(b_1)\ldots f_w(b_n) = g f_w(\alpha_2)f_w(\alpha_1) h $$
where $hg=f_w(a_k)$.  By Lemma \ref{lemma:lineup}, we must have $g=e$ or $h=e$.  Suppose $g=e$.  (The $h=e$ case is similar.)  Then $f_w(\beta)=f_w(\alpha_2\alpha_1 a_k)$, and so again by Lemma \ref{lemma:lineup}, $\beta=\alpha_2\alpha_1 a_k$, which is a cyclic permutation of $\alpha$.
\end{proof}

\begin{lemma}\label{smallcancel}
Let $T\in Tr(2)$.  If $R_T$ denotes the symmetrization of the defining relations for $G_T$, then $R_T$ satisfies the $C'(1/8)$ small cancellation condition.
\end{lemma}
\begin{proof}
We only need to check the positive relations, since they satisfy the $C'(1/8)$ condition iff their inverses do, and there is no overlap between a positive word and a negative word.

We begin with an easy case.  Suppose that $w\in 2^{<\N}$ and consider $f_w(x)^{n_x}$ and $f_w(y)^{n_y}$, where $n_x$ and $n_y$ denote the appropriate exponent, which depends on whether $w\in T$.  As $f_w(x)$ and $f_w(y)$ do not start with the same letter, they do not have a common initial segment.  We must also consider common initial segments of cyclic permutations of these two words, since we had to add the cyclic permutations of $f_w(x)^{n_x}$ and $f_w(y)^{n_y}$ to $R_T$ to make sure that it was symmetrized.

A picture of sorts helps in the analysis.  Before any sort of cyclic permutation, the two words can naturally be seen as being split into $f_w$-blocks.  When a word is cyclically permuted a bit, the blocks at the beginning and end are truncated, as in figure \ref{blocksshifted}.  Now we cannot determine which word is a power of $f_w(x)$ and which is a power of $f_w(y)$ just by looking at the first letter of the words as before.

\begin{figure}
$$yx^5yx^5yx^5yy^5x \underbrace{ \underbrace{x^5yx^5yx^5yx^5yx^5yy^5x}_{f_{00}(x)} \ldots \underbrace{x^5yx^5yx^5yx^5yx^5yy^5x}_{f_{00}(x)}}_{n_x-1\,\mathrm{times}} x^5yx^5$$
\caption{$f_{00}(x)^{n_x}$ after being cyclically permuted}\label{blocksshifted}
\end{figure}

Let $r_1$ be a cyclic permutation of $f_w(x)^{n_x}$ and $r_2$ be a cyclic permutation of $f_w(y)^{n_y}$ and let $M=\min\{|r_1|,|r_2|\}$.  Suppose that $r_1=st_1$ and $r_2=st_2$ with $s$ maximal.  Permuting both $r_1$ and $r_2$ leftwards by $<|f_w(x)|$ letters, we get $r_1^*$ and $r_2^*$, with $r_1^*=\nolinebreak f_w(x)^{n_x}$.  If the $f_w$-blocks in $r_2^*$ also line up correctly, i.e.\ $r_2^*=f_w(y)^{n_y}$, then we know that $r_1^*$ and $r_2^*$ disagree at their first letter.  Thus $|s|<|f_w(x)|<\frac{1}{8} M$ if $r_1\neq r_2$.

Suppose that $r_2^*$ is not in alignment, i.e.\ it is not $f_w(y)^{n_y}$.  Then for $r_1^*$ and $r_2^*$ to agree for $\geq |f_w(x)|$ letters, $f_w(x)$ must be a subword of $f_w(y)^2$ containing letters from each copy of $f_w(y)$.  But this is not possible, by Lemma \ref{lemma:lineup}.  Thus $r_1^*$ and $r_2^*$ agree for $<|f_w(x)|$ letters if they are out of alignment and so $|s|<2|f_w(x)|<\frac{1}{8}M$.  In fact, the same reasoning shows that two different cyclic permutations of $f_w(x)^{n_x}$ or of $f_w(y)^{n_y}$ also overlap for less than $\frac{1}{8}M$ letters.

Now we consider the case when $w,v\in T$ are distinct and \mbox{$a,b\in\{x,y\}$}.  Let $r_1$ be a cyclic permutation of $(f_w(a))^{n_a}$ and $r_2$ be a cyclic permutation of $(f_v(b))^{n_b}$, and let $M=\min\{|r_1|,|r_2|\}$.  If $v=e$ and $w\neq e$, then we observe that no cyclic permutation of $(f_w(a))^{n_a}$ agrees with $(f_e(b))^{n_b}=b^{n_b}$ for more than 6 letters, which is less than $1/8$ of the length of either word.  If $w$ and $v$ begin with different symbols, then one of $r_1$ and $r_2$ will be a cyclic permutation of a word in $x^5y$ and $y^5x$, while the other will be a cyclic permutation of a word in $x^2yxyx$ and $y^2xyxy$.  Then the biggest possible common initial segment between $r_1$ and $r_2$ is $xyx^3$ or $yxy^3$, which is less than $1/8$ of the length of either word.

So we may assume that $w$ and $v$ start with the same symbols.  Suppose that $w=u^\frown w'$ and $v=u^\frown v'$, with $u$ maximal.  Taking our cue from the notation for the greatest common divisor, we will write this as $u=(w,v)$.  This should not cause confusion, as there are no ordered pairs (or greatest common divisors!) in what follows.  Then up to some truncated bits at the beginning and end, $r_1$ and $r_2$ are both words in $f_u(x)$ and $f_u(y)$, and so we are in a situation very similar to our first case, except that now $r_1$ and $r_2$ contain a mix of $f_u(x)$- and $f_u(y)$-blocks, rather than just being conjugates of a power of one or the other.  See figure \ref{lesserblocksshifted} for a picture.

\begin{figure}
$$y\underbrace{x^5y}_{f_0(x)} \underbrace{x^5y}_{f_0(x)} \underbrace{x^5y}_{f_0(x)} \underbrace{y^5x}_{f_0(y)} \underbrace{x^5y}_{f_0(x)} \underbrace{x^5y}_{f_0(x)} \underbrace{x^5y}_{f_0(x)} \underbrace{x^5y}_{f_0(x)} \ldots \underbrace{x^5y}_{f_0(x)} \underbrace{x^5y}_{f_0(x)} \underbrace{x^5y}_{f_0(x)} \underbrace{y^5x}_{f_0(y)} \underbrace{x^5y}_{f_0(x)}x^5$$
\caption{$f_{00}(x)^{n_x}$ after being cyclically permuted, with $f_0(x)$- and $f_0(y)$-blocks shown}\label{lesserblocksshifted}
\end{figure}

Suppose that $r_1$ and $r_2$ are both made up entirely of $f_u$-blocks, i.e.\ there are no truncated $f_u$-blocks at their beginning and end.  Because $f_u(x)$ and $f_u(y)$ start with different letters, we see that if $r_1$ and $r_2$ agree on the beginning of a block, then they agree for the entire block.  So the largest common initial segment $s$ which $r_1$ and $r_2$ share is made up of entire $f_u$-blocks.  We have $r_1=f_u(\alpha)$, $r_2=f_u(\beta)$, and $s=f_u(\gamma)$ with $\alpha,\beta,\gamma$ words in $x$ and $y$, and so by Lemma \ref{lemma:lineup}, we find that $\gamma$ is a common initial segment of $\alpha$ and $\beta$.  Furthermore, $r_1=f_u(\alpha)$ is a cyclic permutation of $f_w(a^{n_a})=f_u(f_{w'}(a^{n_a}))$, and so by Lemma \ref{functioncyclic}, $\alpha$ is a cyclic permutation of $f_{w'}(a^{n_a})$.  Similarly, $\beta$ is a cyclic permutation of $f_{v'}(b^{n_b})$.

So $\gamma$ is a common initial segment of a cyclic permutation of $f_{w'}(a^{n_a})$ and a cyclic permutation of $f_{v'}(b^{n_b})$.  This brings us back to the earlier cases.  If $w'$ and $v'$ are both nontrivial words, then they start with different symbols, which implies that $|\gamma|\leq 5$, and so $s$ is made up of at most 5 $f_u$-blocks.  On the other hand, $r_1$ and $r_2$ are made up of at least $6\cdot\min\{|f_{w'}(a^{n_a})|,|f_{v'}(b^{n_b})|\}$ $f_u$-blocks, and so $|s|<\frac{1}{8}M$.  If $w'=e$ and $v'\neq e$ or vice versa, then $|\gamma|\leq 6$ and we still find that $|s|<\frac{1}{8}M$.  If $w'=e$ and $v'=e$, then $\gamma$ is empty unless $a=b$, which implies that $r_1=r_2$.

This leaves only the out of alignment cases to deal with.  As before, we may permute $r_1$ and $r_2$ leftwards by $<|f_u(x)|$ letters to get $r_1^*$ and $r_2^*$, with $r_1^*$ a product of $f_u$-blocks.  If $r_2^*$ is not a product of $f_u$-blocks, then Lemma \ref{lemma:lineup} tells us that $r_1^*$ and $r_2^*$ agree for $<|f_u(x)|$ letters, and so $r_1$ and $r_2$ agree for $<2|f_u(x)|$ letters, which is $<1/8$ of the length of each word.  If $r_2^*$ is a product of $f_u$-blocks, then we are in the previous case, and we have seen that either $r_1^*=r_2^*$ or the corresponding common initial segment between them consists of at most $6$ $f_u$-blocks.  This implies that $|s|<7|f_u(a)|<\frac{1}{8}M$.  We have finally exhausted all of the cases and have shown that $R_T$ satisfies the $C'(1/8)$ condition, as desired.
\end{proof}

\begin{lemma}\label{homomorphism}
If $T,T'\in Tr(2)$ and there exists $w\in 2^{<\N}$ such that $T=T'_w$, then $G_T\into G_{T'}$.
\end{lemma}
\begin{proof}
This is obvious if $w=e$, and so we may assume that $w$ is nontrivial.  It is easy to see that $f_w$, viewed as a map from $G_T$ to $G_{T'}$, is a homomorphism, since it will take defining relations in $G_T$ to defining relations in $G_{T'}$.  In more detail, 
\begin{align*}
f_v(x)^{59},f_v(y)^{61}\in R_T &\LongIff v\in T \\
 &\LongIff w^\frown v\in T' \\
 &\LongIff f_{w^\frown v}(x)^{59}=f_w(f_v(x)^{59}), \\
 & \hspace{15pt} f_{w^\frown v}(y)^{61}=f_w(f_v(y)^{61}) \in R_{T'}
\end{align*}
and similar equivalences hold for $f_v(x)^{67}$ and $f_v(y)^{71}$.  It remains to show that $f_w$ is an embedding.

We still need to show that nontrivial elements of $G_T$ do not map to the identity in $G_{T'}$.  As in the previous section, we will show that if $\alpha\in\F_2$ is such that $f_w(\alpha)$ contains more than $1/2$ of a relation in $R_{T'}$, then $\alpha$ contains more than $1/2$ of a relation in $R_T$, which easily implies the result.

Suppose that $\alpha\in\F_2$ and that $f_w(\alpha)=1$ in $G_{T'}$.  Then $f_w(\alpha)$ contains more than half of a relation $r\in R_{T'}$.  We know that $r$ is a cyclic permutation of some $f_v(a^{n_a})$, where $v\in 2^{<\N}$, $a\in\{x^{\pm 1}, y^{\pm 1}\}$, and $n_a\in\{59,61,67,71\}$.  Let $u=(w,v)$, so that $w=u^\frown w'$ and $v=u^\frown v'$.  Then $f_w(\alpha)=f_u(f_{w'}(\alpha))$, and $r$ is a cyclic permutation of $f_u(f_{v'}(a^{n_a}))$.  By assumption, the subword of $r$ that both words contain must be big enough to contain an entire $f_u$-block.  Thus Lemma \ref{lemma:lineup} tells us that the $f_u$-blocks of $r$ and $f_w(\alpha)$ must line up.  The $f_u$-blocks are uniquely identified by their first or last letters, so once $f_w(\alpha)$ and $r$ agree for part of an $f_u$-block, they agree on the whole thing, unless $r$ begins and ends with a truncated $f_u$-block.  In this case, cyclically permuting $r$ until it is made up of $f_u$-blocks will ``complete" the $f_u$-block at one end of $r$.  This new word is also a relation which agrees with $f_w(\alpha)$ for at least as long as $r$ did, since either only one end of $r$ was in $f_w(\alpha)$ and cyclically permuting increases the length of the word the two agree on, or $r$ was a subword of $f_w(\alpha)$ and this cyclic permutation is also a subword of $f_w(\alpha)$.

Thus we may assume that $r=f_u(\omega)$ for some $\omega\in\F_2$, and that $r$ and $f_w(\alpha)$ share a subword of the form $f_u(\gamma)$, where $\gamma\in\F_2$.  By Lemma \ref{functioncyclic}, we know that $\omega$ is a cyclic permutation of $f_{v'}(a^{n_a})$.  Then $\gamma$ is a subword of a cyclic permutation of $f_{v'}(a^{n_a})$ and a subword of $f_{w'}(\alpha)$.  If $w'$ and $v'$ are both nontrivial, then they begin with different symbols, and so $|\gamma|\leq 5$.  But then
\begin{align*}
\frac{|f_u(\gamma)|}{|f_v(a^{n_a})|} &= \frac{|\gamma|}{|f_{v'}(a^{n_a})|} \\
 &\leq  \frac{5}{|f_{v'}(a^{n_a})|} \\
 &< 1/2
\end{align*} 
which is a contradiction.  If $v'=e$ but $w'\neq e$, then virtually the same inequalities hold since $f_w'(\alpha)$ does not contain any letter to a power greater than 6, and again we get a contradiction.  Thus $w'=e$, meaning $w=u$, and so $\alpha$ contains $>1/2$ of a cyclic permutation of $f_{v'}(a^{n_a})$.  Further, since $w^\frown v'\in T' \Leftrightarrow v'\in T$, it follows that $f_{v'}(a^{n_a})\in R_T$.  Thus if $f_w(\alpha)=1$ in $G_{T'}$, then $\alpha$ contains $>1/2$ of a word in $R_T$, as desired.
\end{proof}

The proof of the converse of Lemma \ref{homomorphism} will depend on the following two lemmas.

\begin{lemma}\label{cyclesubword}
Suppose $\alpha,\beta\in\F_2$ are cyclically reduced, and $w,v\in 2^{<\N}$.  If $r_1$ is a cyclic permutation of $f_w(\alpha)$ and $r_2$ is both a cyclic permutation of $f_v(\beta)$ and a subword of $r_1$, then $v\subset w$ or $w\subset v$.

Moreover, if $w=v^\frown w'$ then a cyclic permutation of $f_{w'}(\alpha)$ contains a cyclic permutation of $\beta$, and if $v=w^\frown v'$, then a cyclic permutation of $\alpha$ contains a cyclic permutation of $f_{v'}(\beta)$.
\end{lemma}
\begin{proof}
The result is trivial if $w=e$ or $v=e$, so we may assume that $w$ and $v$ are nontrivial.  Let $u=(w,v)$, with $w=u^\frown w'$ and $v=u^\frown v'$.  If $u=e$, then $w$ and $v$ begin with different symbols, which is impossible, since $|r_1|,|r_2|>5$, the length of the longest possible agreement between $r_1$ and $r_2$.  So $u$ is nontrivial, and $r_1$ is a cyclic permutation of $f_u(f_{w'}(\alpha))$, while $r_2$ is a cyclic permutation of $f_u(f_{v'}(\beta))$.  The $f_u$-blocks of each word must line up, by Lemma \ref{lemma:lineup}.  Further, since $r_2$ is a subword of $r_1$, any truncated bits of $f_u$-blocks at the ends of $r_2$ are duplicated in $r_1$.  So we can permute $r_1$ and $r_2$ the same amount to get $r_1^*=f_u(\gamma)$ and $r_2^*=f_u(\omega)$, words composed entirely of $f_u$-blocks, with $r_2^*$ contained in $r_1^*$.  By Lemma \ref{functioncyclic}, we know $\gamma$ is a cyclic permutation of $f_{w'}(\alpha)$ and $\omega$ is a cyclic permutation of $f_{v'}(\beta)$, and that $\omega$ is a subword of $\gamma$.

If $w'$ and $v'$ are nontrivial, then they start with different symbols, and as above we reach a contradiction.  Thus either $w'=e$, and so $w\subset v$ and a cyclic permutation of $\alpha$ contains a cyclic permutation of $f_{v'}(\beta)$, or $v'=e$, so $v\subset w$ and a cyclic permutation of $f_{w'}(\alpha)$ contains a cyclic permutation of $\beta$.
\end{proof}

\begin{lemma}\label{lemma:composed}
Suppose that $t,u,v\in 2^{<\N}$, and some cyclic permutation of $f_t(x^k)$ is a product of a cyclic permutation of $f_u(x^l)$ and a cyclic permutation of $f_v(y^m)$, with $k,l,m\in\Z\setminus\{0\}$.  Then $u=v$, $t=u^\frown 0$, $k=m=\pm 1$, $l=5k$, and \mbox{$f_t(x^k)=f_{u^\frown 0}(x^\pm 1)$} is either $$f_u(x^5)f_u(y)$$ or
$$f_u(y^{-1})f_u(x^{-5}). $$
\end{lemma}
\begin{proof}
By Lemma \ref{cyclesubword}, either $t\subset u$ or $u\subset t$.  If $t\subset u$ and $u=t^\frown u'$, then by the previous lemma, we find that a cyclic permutation of $x^k$ contains a cyclic permutation of $f_{u'}(x^l)$.  This is impossible unless $u'=e$.  So we may assume that $u\subset t$ and $t=u^\frown t'$.  Similarly we find that $v\subset t$ and $t=v^\frown t''$.  It follows that $u\subset v$ or $v\subset u$.

Suppose that $u\subset v$ and $v=u^\frown v'$.  We know that the $f_u$-blocks in $f_u(x^l)$ and $f_u(f_{v'}(y^m))$ must line up with those in $f_t(x^k)$.  This means in particular that a truncated $f_u$-block at the end of the cyclic permutation of $f_u(x^l)$ must be completed by a truncated $f_u$-block at the beginning of the cyclic permutation of $f_u(f_{v'}(y^m))$, and vice versa.  So we can assume that the cyclic permutations we are considering are made up of complete $f_u$-blocks.  Then by Lemma \ref{lemma:lineup} we obtain that a cyclic permutation of $f_{t'}(x^k)$ is a product of $x^l$ and a cyclic permutation of $f_{v'}(y^m)$.  Now, $f_{t'}(x^k)$ is composed of $f_{v'}$-blocks, which must line up with those in the cyclic permutation of $f_{v'}(y^m)$.  So it must be the case that $x^l$ is a cyclic permutation of $f_{v'}$-blocks.  But this can only happen if $v'=e$, meaning $u=v$.  Similar reasoning applies if $v\subset u$.  Thus $u=v$.

It is not possible for a truncated $f_u(x)$-block to be completed by a truncated $f_u(y)$-block, or vice versa, and so we must have that the cyclic permutation of $f_t(x^k)$ we started with is either $f_u(x^l)f_u(y^m)$ or $f_u(y^m)f_u(x^l)$.

It follows that $x^l y^m$ is a cyclic permutation of $f_{t'}(x^k)$.  This can only happen if $t'=0$ and $k,l,m$ are as in the statement of the lemma, since if $t'=e$ then $x^k$ contains no occurrences of $y$, and if $t'\neq 0$ is nontrivial then $f_{t'}(x^k)$ must contain at least two distinct blocks of $x$'s and $y$'s.
\end{proof}

We now take up the converse of Lemma \ref{homomorphism}, which will complete the proof of Theorem \ref{thm:FGGroupEmbedUniversal}.

\begin{lemma}\label{treegroupembed}
If $T,T'\in Tr(2)$ and $G_T\into G_{T'}$, then $\exists w\in 2^{<\N}$ such that $T=T'_w$.
\end{lemma}
\begin{proof}
Suppose that $\theta\from G_T\to G_{T'}$ is a monomorphism.  Our main goal is to prove that $\theta$ must actually be $f_w$ for some $w\in 2^{<\N}$, up to an inner automorphism of $G_{T'}$.  Once we know this, it is easy to recover the relations in each group, and thus to show that $T=T'_w$.

Since $x=f_e(x)$, $x$ has some finite order $n_x$ in $G_T$.  Then $\theta(x)$ has order $n_x$, and so by Theorem \ref{thm:torsion}, $\theta(x)$ must be conjugate to a power of some $f_w(x)$, where $w\in T' \Iff e\in T$.  If we follow $\theta$ by an inner automorphism of $G_{T'}$, we may assume that $\theta(x)=(f_w(x))^\delta$ for some nonzero integer $\delta$.

Similarly, $\theta(y)$ is conjugate to a power of some $f_v(y)$ with $v\in T' \Iff e\in T$.  We find that $\theta(y)=u (f_v(y))^\gamma u^{-1}$.  We can assume that $u$ does not contain more than half of an element of $R_{T'}$.  We may also follow $\theta$ by the inner automorphism corresponding to $f_w(x)$ as necessary to ensure that $u$ does not begin with a power of $f_w(x)$, and this will not change the value of $\theta(x)$.  After freely reducing we get that $\theta(y)=u'r{u'}^{-1}$, where $r$ is a cyclic permutation of $(f_v(y))^\gamma$.  To see this, suppose that $u=\alpha\beta^{-1}$, where $(f_v(y))^\gamma=\beta\mu$ and $\beta$ is the longest subword of $u$ for which this is true.  Then $u(f_v(y))^\gamma u^{-1} = \alpha\beta^{-1}\beta\mu\beta\alpha^{-1} = \alpha\mu\beta\alpha^{-1}$.  Then $u'=\alpha$ and $r=\mu\beta$.  A similar argument works if $(f_v(y))^\gamma$ cancels with $u^{-1}$.  For example, if we had $\theta(y)=xy^{-1} f_0(y) yx^{-1}=xy^{-1} (y^5 x) yx^{-1}$, then after freely reducing we would find $\theta(y)=x y^4 x y x^{-1}=x r x^{-1}$, where $r=y^4xy$.

We now proceed much as in the proof of \ref{thm:EmGroupsUniversal}.  Let $m_x$ be the order of $f_0(x)$ in $G_T$.  Since $\theta$ is an embedding, it must take $f_0(x)$ to a torsion element of order $m_x$.  We know that
\begin{align*}
\theta(f_0(x)) &= \theta(x^5y)=(f_w(x))^{5\delta} u' r {u'}^{-1} \\
 &= (f_w(x))^{\delta'} u' r {u'}^{-1}
\end{align*}
where $|\delta'|< \lfloor \frac{n_x}{2} \rfloor$.  Note that, as written, this word may not be freely reduced, so we can not necessarily use Theorem \ref{thm:wordproblem} yet.  Let $z=\theta(f_0(x))$.

Suppose that $z$ is cyclically reduced as written and that $u'\neq 1$.  Then $z$ is a cyclically reduced word with finite order in $G_{T'}$ which contains a mixture of positive and negative letters, which is impossible by Theorem \ref{thm:torsion}, since the words in $R_{T'}$ are either entirely positive or entirely negative.  So either $u'=1$ or else $z$ is not cyclically reduced.

First suppose that $u'=1$, so that $z=(f_w(x))^{\delta'} r$.  If $\delta'$ and $\gamma$ have the same sign, then $z$ is cyclically reduced as written.  By Theorem \ref{thm:torsion}, $z$ is a cyclic permutation of some $f_t(x^k)$, and so by Lemma \ref{lemma:composed}, $\theta(x)=f_w(x)^{\pm 1}$ and $\theta(y)=f_w(y)^{\pm 1}$.

Suppose that $\delta'$ and $\gamma$ have opposite signs.  There is a (possibly trivial) inner automorphism $\psi$ such that $\psi(z)=sr'$, where $s$ is a cyclic permutation of $f_w(x)^{\delta'}$, and $r'$ is a cyclic permutation of $f_v(y)^\gamma$, and freely reducing $sr'$ will leave a cyclically reduced word.  Let $\psi(z)=z'$ with $z'$ the result of freely reducin $sr'$.  Since $z'$ is a torsion element, it must be a cyclic permutation of some $f_t(x^k)$, and so its letters must all have the same sign.

Suppose $z'$ and $s$ have letters of the same sign.  Then $z'{r'}^{-1}$ is cyclically reduced and so we have a cyclic permutation of $f_w(x)^{\delta'}$ written as a product of a cyclic permutation of $f_t(x)^k$ and a cyclic permutation of $f_v(y)^\gamma$.  By Lemma \ref{lemma:composed}, we get that $\theta(x)=f_{v^\frown 0}(x^{\pm 1})=f_v((x^5y)^{\pm 1})$, and $\theta(y)=f_v(y^{\mp 1})$.  But then either $\theta(xy)=f_v(x^5)$ or $\theta(xy)=f_v(y^{-1}x^{-5}y)=f_v(y)^{-1}f_v(x^{-5})f_v(y)$.  Both of these are torsion elements, but $xy$ is not a torsion element in $G_T$, which contradicts the fact that $\theta$ is an embedding.  So suppose that $z'$ and $r'$ have letters of the same sign.  Then $z's^{-1}$ is cyclically reduced and so we have a cyclic permutation of $f_v(y)^\gamma$ written as a product of a cyclic permutation of $f_t(x)^k$ and a cyclic permutation of $f_w(x)^{\delta'}$.  Arguing as in the proof of Lemma \ref{lemma:composed}, we find that $w=t$ and that $w\subset v$.  Let $v=w^\frown v'$.  We obtain that $f_{v'}(y^\gamma)=x^{-\delta'+k}$, which is absurd.

We still must address the case where $u'\neq 1$ and $z$ is not cyclically reduced.  This can happen for two reasons.  It may be that $u'$ and $f_w(x)$ begin in the same way.  We know that $u'$ does not begin with an entire copy of $f_w(x)$, and we have assumed that $u'$ does not further cancel with $r$, so there is an inner automorphism $\psi$ such that $\psi(\theta(f_0(x)))=su''r{u''}^{-1}$, a cyclically reduced word with $s$ a cyclic permutation of $(f_w(x))^{\delta'}$.  If $u''\neq 1$, then $\psi(\theta(f_0(x))$ contains positive and negative letters, and we have already seen that this is impossible.  Thus we must have that $\psi(\theta(f_0(x)))=sr$, and as before we see that $\psi(\theta(x))=f_w(x)^{\pm 1}$ and $\psi(\theta(y))=f_w(y)^{\pm 1}$.

The other possibility is that $u'$ cancels with the end of $f_w(x)$.  It cannot cancel with the whole of $f_w(x)$, and so again after following $\theta$ by an appropriate inner automorphism $\psi$ we get that $\psi(\theta(f_0(x)))=su''r{u''}^{-1}$, a cyclically reduced word with $s$ a cyclic permutation of $(f_w(x))^{\delta'}$.  This case is treated exactly as in the previous paragraph.

So we have shown that, up to an inner automorphism, $\theta(x)=f_w(x)^{\pm 1}$ and $\theta(y)=f_w(y)^{\pm 1}$ with the signs matching.  If $\theta(x)=f_w(x)$ and $\theta(y)=f_w(y)$, then $\theta=f_w$, and hence
\begin{align*}
u\in T &\Leftrightarrow f_u(x^{53})\in R_T \\
 &\Leftrightarrow f_w(f_u(x^{53}))\in R_{T'} \\
 &\Leftrightarrow w^\frown u\in T'
\end{align*}

Thus $T=T'_w$, as desired.  Thus it only remains is to eliminate the undesirable case when $\theta(x)=f_w(x^{-1})$ and $\theta(y)=f_w(y^{-1})$.  In this case we have that

\begin{align*}
\theta(f_{00}(x)) &=\theta(x^5yx^5yx^5yx^5yx^5yy^5x) \\
 &=f_w(x^{-5}y^{-1}x^{-5}y^{-1}x^{-5}y^{-1}x^{-5}y^{-1}x^{-5}y^{-1}y^{-5}x^{-1}) \\
 &=f_w((xy^5yx^5yx^5yx^5yx^5yx^5)^{-1})
\end{align*}
We will show this is not a torsion element in $G_{T'}$.  Since $f_{00}(x)$ is a torsion element in $G_T$, this implies that $\theta$ is not a homomorphism, which is a contradiction.

Let $\alpha=xy^5yx^5yx^5yx^5yx^5yx^5$.  It is easy to see that $f_{00}(x)$ is the only torsion element that has length 36 and that contains 26 $x$s.  However, $\alpha$ is not a cyclic permutation of $f_{00}(x)$.  It follows that $f_w(\alpha)$ (and hence $\theta(f_{00}(x))$) is not a torsion element, since otherwise it would have to be a cyclic permutation of some $f_t(x^k)$ with $t\in T'$.  Thus by Lemma \ref{cyclesubword}, either $w\subset t$ or $t\subset w$.  Suppose that $w\subset t$ and $t=w^\frown t'$.  Then $\alpha$ must be a cyclic permutation of $f_{t'}(x^k)$, which we have already seen is impossible.  If $t\subset w$ and $w=t^\frown w'$, then $x^k$ must be a cyclic permutation of $f_{w'}(\alpha)$, which is also impossible.  Thus we reach a contradiction, eliminating the final undesirable case.
\end{proof}

\begin{proof}[Proof of Theorem \ref{thm:FGGroupEmbedUniversal}]
Lemmas \ref{homomorphism} and \ref{treegroupembed} establish that the map $T\mapsto G_T$ is a Borel reduction from $\qlt_2^{tree}$ to $\qlt_{em}$.
\end{proof}

\begin{cor}
$\equiv_{em}$ is a universal countable Borel equivalence relation.
\end{cor}

It may have been possible to prove this corollary without any reference to quasi-orders, by reducing some known universal countable Borel equivalence relation to $\equiv_{em}$.  However, it seems that the most natural and direct route to this result is through Theorem \ref{thm:FGGroupEmbedUniversal}.  A closer look at the above proof also leads to the following result, which tells us that the bi-embeddability relation on the groups constructed above is much more complicated than the isomorphism relation on these groups.

\begin{cor}
With the above notation, $G_T\cong G_S \Iff T=S$.
\end{cor}
\begin{proof}
One direction is trivial.  For the other, suppose that $T,S\in Tr(2)$ are such that $G_T\cong G_S$, via $\phi\from G_T\to G_S$.  Then in particular $\phi$ is an embedding, and by the proof of Lemma \ref{treegroupembed}, there exists $w\in 2^{<\N}$ such that $T=S_w$.  Furthermore, after adjusting by an inner automorphism of $G_S$ if necessary, we can suppose that $\phi=f_w$.  We will show that if $w\neq e$, then $f_w\from G_T\to G_S$ is not surjective.  Hence $w=e$ and $G_T=G_S$.

Suppose that $w\in 2^{<\N}$ is nontrivial and that $f_w\from G_T\to G_S$ is surjective.  Then there is some word $\alpha\in \F_2$, which we may assume does not contain more than $1/2$ of a relation in $R_T$, such that $f_w(\alpha)=x$ in $G_S$, where $x$ is one of the generators of $G_S$.  This means that $f_w(\alpha)x^{-1}=1$ in $G_S$.  By the proof of Lemma \ref{homomorphism}, we know that $f_w(\alpha)$ does not contain more than $1/2$ of a relation in $G_S$.  By Theorem \ref{thm:wordproblem}, for $f_w(\alpha)x^{-1}$ to represent the identity in $G_S$, it must contain more than $(1-3/8)=5/8$ of a relation in $R_T$.  But $f_w(\alpha)x^{-1}$ has at most one more letter in common with a relation than $f_w(\alpha)$ does.  Since $f_w(\alpha)$ contains less than $1/2$ of a relation in $R_T$, $f_w(\alpha)x^{-1}$ contains less than $5/8$ of a relation in $R_T$.  This is a contradiction.
\end{proof}

\section{Quasi-orders that symmetrize to universal countable Borel equivalence relations}\label{sec:Final}

We have seen that symmetrizing universal countable Borel quasi-orders produces universal countable Borel equivalence relations.  It would be desirable to be able to strengthen this and say that symmetrizing a non-universal countable Borel quasi-order produces a non-universal countable Borel equivalence relation.  However, it is easy to see that this is not the case.  Let $E$ be a universal countable Borel equivalence relation.  Then $E$ is not universal as a quasi-order, as symmetry is preserved downwards under $\leq_B$, but obviously $E$ symmetrizes to a universal equivalence relation, i.e.\ $E$ itself.

If this were the extent of the phenomenon, then it would still be possible to use negative results about the universality of a given countable Borel quasi-order in order to prove negative results about the universality of its associated equivalence relation, so long as the quasi-order was asymmetric somewhere.  Unfortunately, things are as bad as they could be.

\begin{theorem}\label{thm:Final}
There are $2^{\aleph_0}$ countable Borel quasi-orders $Q$, distinct up to Borel bireducibility, for which $E_Q$ is a universal countable Borel equivalence relation.
\end{theorem}

\begin{proof}
Recall that Adams and Kechris showed in \cite{AK00} that there are $2^{\aleph_0}$ countable Borel equivalence relations up to Borel bireducibility.  The equivalence relations they produced are all defined on different spaces; we will write $X_E$ for the space on which $E$ is defined.  Further, these equivalence relations have all equivalence classes of size at least 2, a minor technical point that will be used later.

Given a countable Borel equivalence relation $E$ on a standard Borel space $X$ with a Borel linear order $\leq$, define $E(\leq)$ to be $E\hspace{3pt}\cap \leq$.  We see that $E(\leq)$ is an asymmetric countable Borel quasi-order, unless $E$ is just equality on $X$, written $\Delta(X)$.  Note that
$$ E(\leq) \leq_B F(\leq) \;\;\Longrightarrow\;\; E \leq_B F $$
and so if $E(\leq)$ and $F(\leq)$ are Borel bireducible, then so are $E$ and $F$.  Thus by the result of Adams and Kechris we find that there are $2^{\aleph_0}$ quasi-orders of the form $E(\leq)$ up to Borel bireducibility.

Now we define the family of quasi-orders $\{ E(\leq)\vee E_\infty\}$, where $E$ ranges over the equivalence relations from the proof of Adams and Kechris and $\vee$ denotes the disjoint union of two equivalence relations.  Now suppose that
$$ E(\leq)\vee E_\infty \,\, \leq_B \,\, F(\leq)\vee E_\infty. $$
Then $E(\leq)\leq_B F(\leq)$, since $E(\leq)$ is completely asymmetric, and so any two $E(\leq)$-comparable elements in $X_E$ must map to $X_F$.  Since the equivalence classes of $E$ are all of size at least 2, every element of $X_E$ is $E(\leq)$-comparable with some other element of $X_E$, and so every element of $X_E$ is mapped to $X_F$.

Thus there are $2^{\aleph_0}$ countable Borel quasi-orders of the form $E(\leq)\vee E_\infty$.  Each of these symmetrizes to $\Delta(X_E)\vee E_\infty$, which is universal.
\end{proof}

Most of the reducibilities from computability theory, such as Turing reducibility $\leq_T$ or 1-reducibility $\leq_1$, are countable Borel quasi-orders.  The equivalence relations associated with them have been the subject of a great deal of work in descriptive set theory; for example, see \cite{DK99}.  It remains open whether $\equiv_T$ or $\equiv_1$ are universal countable Borel equivalence relations.  Thus the following question naturally arises.

\begin{problem}
Are any of $\leq_1,\leq_T$, etc.\ universal countable Borel quasi-orders?
\end{problem}

In light of the previous theorem, showing that $\leq_T$ or $\leq_1$ are not universal countable Borel quasi-orders would not settle the question of whether $\equiv_T$ or $\equiv_1$ are universal, although it may be taken as ``evidence" that the answer is no.

\section{Acknowledgements}
The results in this paper (except for Theorem \ref{thm:Final}) are from my thesis, which was done under the direction of Simon Thomas, whose help was invaluable.  I would also like to thank Arthur Apter, Justin Bush, Gregory Cherlin, David Duncan, and Charles Weibel for many helpful mathematical discussions.

\bibliographystyle{plain}
\bibliography{bibliography}{}

\end{document}